\documentclass[final,12pt,times]{elsarticle}

\usepackage{booktabs} 
\usepackage{epsfig}
\usepackage{amsfonts}
\usepackage{amssymb}
\usepackage{amsmath}
\usepackage{amssymb}
\usepackage{amsthm}
\usepackage{graphics}
\usepackage{graphicx}
\usepackage{subcaption}
\usepackage{mathrsfs}
\usepackage{indentfirst}
\usepackage{stmaryrd}
\usepackage{latexsym}
\usepackage{xypic}
\usepackage{tikz}
\usepackage{hyperref}
\usepackage{listings}
\usepackage{algorithm}
\usepackage{algorithmic}
\usepackage{tikz-cd}
\usepackage{amsmath,amssymb}
\usepackage{CJK}
\biboptions{sort&compress}
\theoremstyle{plain}
\theoremstyle{definition}
\newtheorem{definition}{Definition}[section]
\newtheorem{lemma}[definition]{Lemma}
\newtheorem{theorem}[definition]{Theorem}
\newtheorem{proposition}[definition]{Proposition}

\newtheorem{remark}[definition]{Remark}
\newtheorem{corollary}[definition]{Corollary}
\theoremstyle{break}

\journal{Studia Logica}

\begin{document}

\begin{frontmatter}


\title{\textbf{A categorical equivalence for monadic algebras of first-order substructural logics motivated by Kalman's construction}}


\author{Juntao Wang$^{a}$, Mei Wang$^{b,*}$, William Zuluaga Botero$^{c}$}
\cortext[cor1]{Corresponding author. \\
Email addresses: wjt@xsyu.edu.cn (J.T. Wang), wangmeimath@163.com(M. Wang),\\ wizubo@gmail.com (W. Zuluaga Botero).}

\address [A] {School of Science, Xi'an Shiyou University, Xi'an 710065, Shaanxi, P.R. China}
\address[B]{School of Science, Xi Hang University, Xi'an 710077, Shaanxi, P.R. China}
\address[C]{ CONICET and Departamento de Matematica, Facultad de Ciencias Exactas, Universidad Nacional del Centro de la Provincia de Buenos Aires, Pinto, 399, 7000, Tandil, Argentina}

\begin{abstract}
The category $\mathbb{DRDL'}$, whose objects are c-differential residuated distributive lattices that satisfy the condition $\mathbf{CK}$, is the image of the category $\mathbb{RDL}$, whose objects are residuated distributive lattices, under the categorical equivalence (Kalman functor) $\mathbf{K}$. The main goal of this paper is to lift this equivalence $\mathbf{K}$ to the category $\mathbb{MRDL}$, whose objects are monadic residuated distributive lattices, and the category $\mathbb{MDRDL'}$, whose objects are pairs formed by an object of $\mathbb{DRDL'}$ and a center universal quantifier. Firstly, based on the variety of monadic FL$_\textrm{e}$-algebras, we introduce the concept of monadic residuated lattices and study some of their further algebraic properties, proving the classes of monadic residuated distributive lattices and monadic c-differential residuated distributive lattices are in one-to-one correspondence. Subsequently, based on this corresponding relation, we prove that there exists a categorical equivalence between the categories $\mathbb{MRDL}$ and $\mathbb{MDRDL'}$. The results of this paper not only generalizes the works of Sagastume and San Mart\'{i}n in [Mathematical Logic Quarterly, {\bf 60}(2014), 375--388], but also addresses and overcomes the limitations identified in the works of [Studia Logica, {\bf 111}(2023), 361--390]. Finally, this paper concludes with some applications regarding descriptions of a 2-contextual translation.
\end{abstract}


\end{frontmatter}

\section{Introduction}
\label{intro}

  Substructural logics form a wide family of nonclassical logics that can be roughly defined as those logic systems such that, when presented by means of a Gentzen-style calculus, lack some of the structural rules (exchange, contraction and weakening), which occur in standard sequent systems of classical and intuitionistic logics \cite{Galatos}. As such they encompass a variety of systems independently developed since mid 20th century, including relevant logics or many-valued logics like monoidal logic (not satisfying contraction), linear logic (which, besides the former two, does not satisfy exchange either).  The  Gentzen system $\mathbf{FL}$ of full Lambek calculus, the weakest logic considered in substructural logics, is obtained from the Gentzen system $\mathbf{LJ}$ for intuitionistic propositional logic by removing three structural rules. As a very important extension of $\mathbf{FL}$, the sequent calculus $\mathbf{FL_{ew}}$ is also obtained from the intuitionistic logic by only deleting the contraction rule. Researchers always call substructural logics over $\mathbf{FL_{ew}}$ or logics without contraction rule \cite{Ono}, although the contraction rule holds in some of them. The class of logics without the contraction rule contains intermediate logics, {\L}ukasiewicz's many-valued logics \cite{Chang} and fuzzy logics in the sense of H\'{a}jek \cite{Hajeka}. Our main technical tool is algebraic in nature and is based on close connections between logics over $\mathbf{FL_{ew}}$ and the variety of residuated lattices \cite{Blount,Metcalfe,Cintulab,Galatosc,Galatosd,Galatose}.

Constructive logic with strong negation \cite{Rasiowa} was created for logical reasons: the
intuitionistic negation does not have good constructive properties. From the beginning, constructive logic with strong negation was closely related to intuitionistic logic. The algebraic counterpart of that logic was first studied by Rasiowa, and further defined as a variety by Brignole and Monteiro \cite{Brignolea,Brignoleb}. They call the algebraic
models Nelson algebras.  The algebraic approach of the relationship between intuitionism and constructive
logic with strong negation appears: a characterization of Nelson algebras as pairs of disjoint elements of Heyting algebras \cite{Vakarelov}. Cignoli \cite{Cignoli} also shows that this characterization can be formalized from a categorical point of view as an adjunction between the categories of Heyting algebras and Nelson algebras, which restricted to centered Nelson algebras becomes an equivalence that is induced by the Kalman functor. Furthermore, Castiglioni, Menni and Sagastume proved in \cite {Castiglionib} that the functor $\mathbf{K}$, motivated by a construction due to Kalman, can be extended to an equivalence between the category $\mathbb{RDL}$ of residuated distributive lattices and the category $\mathbb{DRDL'}$ of c-differential residuated distributive lattices that satisfy the condition $\mathbf{CK}$. More recently, Castiglioni, Lewin and Sagastume proved in \cite{Castiglionic,Sagastume1} that the category $\mathbf{MV^\bullet}$, whose objects are monadic MV-algebras, is the image of $\mathbf{MV}$, whose objects are MV-algebras, by the equivalence $\mathbf{K}$ and extended the construction to $\ell$-groups introducing the new category of monadic $\ell$-groups.

 Monadic (Boolean) algebra $(L,\lozenge)$, in the sense of Halmos \cite{Halmos}, is a Boolean algebra equipped with a closure operator $\lozenge$,  which abstracts algebraic properties of the standard existential quantifier ``for some". The name ``monadic" comes from the connection with predicate logics for languages having one placed predicates and a single quantifier. Subsequently, monadic MV-algebras, the algebraic counterpart of monadic {\L}ukasiewicz predicate logic $\mathbf{M{\L}PC}$, were introduced and studied in \cite{Belluce,Cimadamore,Nola,Rachuneka,Rutledge}; monadic Heyting algebras, the algebraic models of Prior's intuitionistic modal logic $\mathbf{MIPC}$, have been introduced and studied in \cite{Bezhanishvilia,Bezhanishvilib,Bezhanishvilic};  monadic BL-algebras, the algebraic semantics of monadic fragment of basic fuzzy predicate logic $\mathbf{mBL\forall}$, have been introduced and studied in \cite{Castano,Castano1,Castano 2}. Meanwhile, Tuyt introduced and studied the variety of monadic FL$_\textrm{e}$-algebras and their corresponding subvarieties, and showed they are indeed the algebraic semantics for one-variable (monadic) fragments of various first-order substructural logics \cite{Tuyt}. On this basis, Cintula, Metcalfe and Tokuda reached several more interesting and insightful conclusions, and established a general criterion for axiomatizing the one-variable fragments of various first-order logics using superamalgamation, offering an alternative proof-theoretic method for substructural logics with cut-free calculi \cite{Cintula}. It should be noted here that the author of the present paper was the first to integrate monadic algebras with the Kalman functor in an effort to give some representations of these algebras. More specifically, the paper \cite{Wang4} provided a first attempt to relate the category $\mathbb{WMRDL}$ of weak monadic residuated distributive lattices to that the category $\mathbb{WMDRDL'}$ of c-differential residuated distributive lattices with weak universal quantifiers via the Kalman functor $\mathbf{K}$, and to give some representations of weak monadic residuated distributive lattices. However, two aspects of this work could be strengthened. On one hand, the variety of weak monadic residuated lattices lacks a logical foundation from the perspective of algebraic logic. In order to address this, we replace it with the variety of monadic residuated lattices in the sense of Tuyt, which has been proven to form the algebraic semantics for the monadic fragments of first-order substructural logics. On the other hand, the categorical equivalence between the categories $\mathbb{WMRDL}$ and $\mathbb{WMDRDL'}$ is proved by relying on the isomorphisms $\alpha$ and $\beta$ (see Theorem 6.3 in \cite{Wang4}), where
  \begin{center} $\alpha(x)=(x,{\sim} x)$ and $\beta(x)=(\lambda(x),\lambda({\sim} x))$.
  \end{center}
  Nevertheless, it would be more appropriate to prove this key result by using the isomorphisms $\phi$ and $\psi$ (see Lemma 7.6 in \cite{Castiglionib}), where
  \begin{center} $\phi(x)=(x,0)$ and $\psi(x)=(x\cup c, {\sim} x\cup c)$,
  \end{center}
 since $\alpha$ and $\beta$ are applicable only to the category of involutive residuated lattices and their subcategories (see Lemma 7 in \cite{Castiglionic}), a limitation that restricts the generality of the argument. Hence, the first motivation of this paper is to prove the corresponding categorical equivalence in a more reasonable and precise way, based on monadic residuated lattices in the sense of Tuyt.
 
 Subsequently, conducted as a direct continuation of the research presented in \cite{Castiglionic}, Sagastume and San Mart\'{i}n \cite{Sagastume} have also lifted the Kalman functor $\mathbf{K}^\bullet$ from the categories $\mathbb{MV}$ and $\mathbb{MV}^\bullet$  to the categories $\mathbb{MV}_\textrm{U}$, whose objects are UMV-algebras (or MV-algebras with U-operators), and $\mathbb{MV}^\bullet_\textrm{cU}$, whose objects are pairs formed by an object of $\mathbf{K}^\bullet$ and a cU-operator, please see the following diagram:
 
 \begin{center}
\begin{tikzpicture}
  \node[inner sep=2pt] (MV^bullet) at (0,0) {$\mathbb{MV}$};
  \node[inner sep=2pt] (MV) at (0,3) {$\mathbb{MV}_\textrm{U}$ };
  \node[inner sep=2pt] (MDRL) at (3,0) {$\mathbb{MV}^\bullet$};
  \node[inner sep=2pt] (iIRL_0) at (3,3) {$\mathbb{MV}^\bullet_\textrm{cU}$};  
\draw[->] ([yshift=3pt]MV.east) -- ([yshift=3pt]iIRL_0.west) node[midway, above=3pt] {$\mathscr{K}$};
\draw[<-] ([yshift=-3pt]MV.east) -- ([yshift=-3pt]iIRL_0.west) node[midway, below=3pt] {$\mathbf{K}^\bullet$};
\draw[->] ([yshift=3pt]MDRL.west) -- ([yshift=3pt]MV^bullet.east) node[midway, above=3pt] {$\mathscr{K}$};
\draw[<-] ([yshift=-3pt]MDRL.west) -- ([yshift=-3pt]MV^bullet.east) node[midway, below=3pt] {$\mathbf{K}^\bullet$};

\draw[->] (MV) -- (MV^bullet) node[midway, left=2pt] {$\mathbf{F}$};
  \draw[->] (iIRL_0) -- (MDRL) node[midway, right=2pt] {$\mathbf{F}$};
\end{tikzpicture} 
\begin{center}  {\bf{Figure 1.}} Diagram of the functors $\mathbf{K}^\bullet$ and $\mathscr{K}$ relate the categories $\mathbb{MV}_\textrm{U}$ and $\mathbb{MV}^\bullet_\textrm{cU}$.
\end{center}
\end{center}
Note that both MV-algebras and MV$^{\bullet}$-algebras satisfy the law of involution, their corresponding UMV-algebras and cUMV$^{\bullet}$-algebras have parallel definitions. In this case, even though the kalman functor $\mathbf{K}^\bullet$ induce an equivalence of the categories $\mathbb{MV}_\textrm{U}$ and $\mathbb{MV}^\bullet_\textrm{cU}$, it cannot simplify the given category $\mathbb{MV}_\textrm{U}$. Given the close relationship between UMV-algebras and monadic residuated lattices, to be more specific, UMV-algebras are a natural generalization of monadic residuated lattices (see Corollary \ref{corollary3.8}(2)), the second motivation is to generalize this result from the category of UMV-algebras to the case of  monadic residuated lattices. More precisely, much of this work is motivated by results due to Castiglioni, Menni and Sagastume \cite{Castiglionib} relating residuated distributive lattices and c-differential residuated distributive lattices with the condition $\mathbf{CK}$ by the functor $\mathbf{K}$, and the conviction that these can be lifted to the level of monadic residuated distributive lattices.
Let $\mathbb{MRDL}$ be the category of monadic residuated distributive lattices.
Then the main goal of this part is to isolate a reasonable category which completes the diagram:
\begin{center}
\begin{tikzpicture}
  \node[inner sep=2pt] (MV^bullet) at (0,0) {$\mathbb{RDL}$};
  \node[inner sep=2pt] (MV) at (0,3) {$\mathbb{MRDL}$ };
  \node[inner sep=2pt] (MDRL) at (3,0) {$\mathbb{DRDL'}$};
  \node[inner sep=2pt] (iIRL_0) at (3,3) {$\mathbb{?}$};  
\draw[->] ([yshift=3pt]MV.east) -- ([yshift=3pt]iIRL_0.west) node[midway, above=3pt] {$\mathbf{K}$};
\draw[<-] ([yshift=-3pt]MV.east) -- ([yshift=-3pt]iIRL_0.west) node[midway, below=3pt] {$\mathbf{C}$};
\draw[->] ([yshift=3pt]MDRL.west) -- ([yshift=3pt]MV^bullet.east) node[midway, above=3pt] {$\mathbf{K}$};
\draw[<-] ([yshift=-3pt]MDRL.west) -- ([yshift=-3pt]MV^bullet.east) node[midway, below=3pt] {$\mathbf{C}$};

\draw[->] (MV) -- (MV^bullet) node[midway, left=2pt] {$\mathbf{F}$};
  \draw[->] (iIRL_0) -- (MDRL) node[midway, right=2pt] {$\mathbf{F}$};
\end{tikzpicture} 
\begin{center}  {\bf{Figure 2.}} Diagram of the functors $\mathbf{K}$ and $\mathbf{C}$ relate the categories $\mathbb{MRDL}$ and $\mathbb{?}$ .
\end{center}
\end{center}
where $\mathbb{MRDL}\rightarrow \mathbb{RDL}$ is the obvious forgetful functor. The category completing this square should be a c-differential residuated distributive lattice with the condition $\mathbf{CK}$ equipped with a monadic structure in such a way that the two structures interact in a reasonable way. Note that c-differential residuated distributive lattices satisfy the law of involution. Hence, in the definition of the corresponding monadic algebras, it is possible to use only one of the existential or universal quantifiers as primitive, the other being definable as the dual of the one taken as part of the signature. However, definitions of monadic residuated distributive lattices require the introduction of both kinds of quantifiers simultaneously, because these quantifiers are not mutually interdefinable. It is clear that the structure of monadic c-differential residuated distributive lattice $(L,\lozenge)$ is simpler than that of monadic residuated distributive lattice $(L,\square,\lozenge)$, because it has only one of the  existential or universal quantifier. In view of this point, the resulting category of our constructions will have a simpler version than the original one. 

Finally, we observe that the categories $\mathbb{MV}_\textrm{U}$ and $\mathbb{MV}^\bullet_\textrm{cU}$ are subcategories of $\mathbb{MRDL}$ and ?, respectively (see Corollary \ref{corollary5.17} and Theorem 13 in \cite{Castiglionic}). Crucially, these structures are not isolated but are fundamentally linked through a sophisticated categorical mechanism. Specifically, their connections are established via canonical equivalence relations induced by the Kalman functor, which acts as a unifying bridge, preserving essential logical and algebraic properties while mapping between these domains. This functorially-induced framework not only clarifies their hierarchical relationships but also reveals a deeper, intrinsic correspondence between their internal architectures.

This leads directly to the third motivation is to systematically explore the intrinsic connections between the structural diagrams in Figure 1 and Figure 2. Moving beyond mere taxonomy, this investigation seeks to uncover the coherent mathematical narrative underlying both representations, namely how they reflect parallel categorical truths mediated by the Kalman functor. Such an endeavor constitutes a highly non trivial and compelling research topic. Indeed, it is precisely this pivotal question, situated at the intersection of algebraic, logical, and categorical perspectives, that we address in the final section of the paper, where we synthesize these insights into a unified theoretical account. Accordingly, it is our goal to finalize the construction presented in Figure 3.

\begin{center}
\begin{tikzpicture}
  \node[inner sep=2pt] (MV^bullet) at (0,0) {$\mathbb{MV}^\bullet_\textrm{cU}$};
  \node[inner sep=2pt] (MV) at (0,3) {$\mathbb{MV}_\textrm{U}$};
  \node[inner sep=2pt] (MDRL) at (3,0) {$?$};
  \node[inner sep=2pt] (iIRL_0) at (3,3) {$\mathbb{MRDL}$};
  \node[inner sep=2pt] (DRLyipie) at (6,0) {$\mathbb{DRDL'}$};
  \node[inner sep=2pt] (IRL_0) at (6,3) {$\mathbb{RDL}$};
  
  \draw[->] ([xshift=3pt]MV.south) -- ([xshift=3pt]MV^bullet.north) node[midway, right=3pt] {$\mathscr{K}$};
  \draw[<-] ([xshift=-3pt]MV.south) -- ([xshift=-3pt]MV^bullet.north) node[midway, left=3pt] {$\mathbf{K}^\bullet$};
  \draw[->] ([xshift=3pt]MDRL.north) -- ([xshift=3pt]iIRL_0.south) node[midway, right=3pt] {$\mathbf{C}$};
  \draw[<-] ([xshift=-3pt]MDRL.north) -- ([xshift=-3pt]iIRL_0.south) node[midway, left=3pt] {$\mathbf{K}$};
  \draw[->] ([xshift=3pt]DRLyipie.north) -- ([xshift=3pt]IRL_0.south) node[midway, right=3pt] {$\mathbf{C}$};
  \draw[<-] ([xshift=-3pt]DRLyipie.north) -- ([xshift=-3pt]IRL_0.south) node[midway, left=3pt] {$\mathbf{K}$};

  \draw[->] (MV) -- (iIRL_0) node[midway, above=2pt] {inc};
  \draw[->] (iIRL_0) -- (IRL_0) node[midway, above=2pt] {$\mathbf{F}$};
  \draw[->] (MV^bullet) -- (MDRL) node[midway, above=2pt] {inc};
  \draw[->] (MDRL) -- (DRLyipie) node[midway, above=2pt] {$\mathbf{F}$};
\end{tikzpicture}
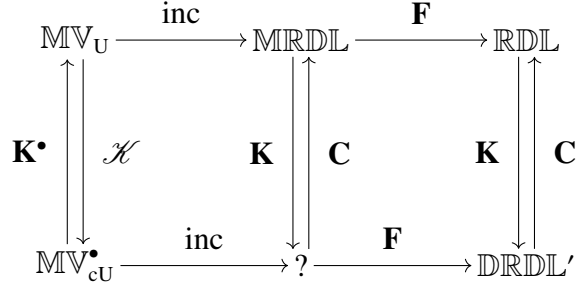
\begin{center}  {\bf{Figure 3.}} Diagram of the relationship between the categories $\mathbb{MV}_\textrm{U}$, $\mathbb{MV}^\bullet_\textrm{cU}$, ? and $\mathbb{MRDL}$.
\end{center}
\end{center}

In order to complete Figures 2 and 3, we introduce the category $\mathbb{MDRDL'}$ of monadic c-differential residuated distributive lattices that satisfy $\mathbf{CK}$ and extend the functor $\mathbf{K}$ to give an equivalence  between the categories $\mathbb{MRDL}$ and $\mathbb{MDRDL'}$, and then complete these two figures. 

The paper is organized as follows: In Section 2, we review some basic results about the categories $\mathbb{RDL}$ and $\mathbb{DRDL'}$. In Section 3, based on the variety of monadic FL$_\textrm{e}$-algebra, we introduce the variety of monadic residuated lattices by the natural semantic expansions. In Section 4, we lift this equivalence $\mathbf{K}$ from  the category $\mathbb{RDL}$ and the category $\mathbb{DRDL'}$ to the category $\mathbb{MRDL}$ and the category $\mathbb{MDRDL'}$. In Section 5, we provide a
a finite, nontrivial $2$-contextual translation from the equational consequence
relation of $\mathbb{MDRDL'}$ into that of $\mathbb{MRDL}$.

\section{Preliminaries}

For the reader's convenience, we include a table summarizing the notation used for classes of algebras and algebraic equations throughout the paper, followed by the corresponding definitions.

\begin{table}[H]
\begin{center}
\begin{tabular}{ccc}
\midrule 
\textbf{Class of algebras} &\textbf{Notation} \\
\midrule
Class of $\mathrm{FL}_\textrm{e}$-algebras & $\mathbb{FL}_\textrm{e}$ \\
Class of residuated lattices & $\mathbb{RL}$ \\
Class of residuated distributive lattices & $\mathbb {RDL}$ \\
Class of c-differential residuated distributive lattices & $\mathbb{DRDL}$\\
Class of  c-differential residuated distributive lattices satisfying $\mathbf{CK}$  & $\mathbb{DRDL'}$\\
Class of MV-algebras & $\mathbb{MV}$\\
Class of MV$^\bullet$-algebras & $\mathbb{MV}^\bullet$ \\
Class of monadic $\mathrm{FL}_\textrm{e}$-algebras & $\mathbb{MFL}_\textrm{e}$\\
Class of monadic residuated lattices & $\mathbb{MRL}$ \\
Class  of monadic involutive residuated lattices&  $\mathbb{MIRL}$ \\
Class of monadic residuated distributive lattices & $\mathbb{MRDL}$ \\
Class of monadic c-differential residuated distributive lattices & $\mathbb{MDRDL}$\\
Class of monadic c-differential residuated distributive lattices satisfying $\mathbf{CK}$  & $\mathbb{MDRDL'}$\\
Class of UMV-algebras & $\mathbb{MV}_\textrm{U}$\\
Class of UMV$^\bullet$-algebras  & $\mathbb{MV}^\bullet_\textrm{cU}$\\
\bottomrule
\end{tabular}
\end{center}
\centering
{\bf{Table 1.}}  Notations for the class of algebras covered throughout the paper.
\end{table}

\begin{table}[H]
\begin{center}
\begin{tabular}{ccc}
\midrule 
\textbf{Algebraic equation} & &\textbf{Name (Notation)} \\
\midrule
$x\leq 1$ & & Integrality $\mathrm{(INT)}$ \\
$(x\rightarrow y)\vee (y\rightarrow x)=1$ & & Prelinearity $\mathrm{(PRE)}$ \\
$x\wedge y=x\odot(x\rightarrow y)$ &&  Divisibility $\mathrm{(DIV)} $ \\
$\neg\neg x=x $ &&  Involution $\mathrm{(INV)}$ \\
\bottomrule
\end{tabular}
\end{center}
\centering
{\bf{Table 2.}} Notations for the algebraic equations covered throughout the paper.
\end{table}

In this section, we summarize some basic results concerning the categories $\mathbb{RDL}$, whose objects are residuated distributive lattices, and $\mathbb{DRDL'}$ , whose objects are c-differential residuated distributive lattices. The category $\mathbb{DRDL'}$ arises as the image of $\mathbb{RDL}$ under the Kalman functor $\mathbf{K}$, which establishes a categorical equivalence between them. Moreover, $\mathbb{MV}^\bullet$ can be regarded as the full subcategory of $\mathbb{DRDL'}$ whose objects are c-differential residuated lattices satisfying the condition $\mathbf{CK}$.

\begin{definition}\label{definietion2.1}\cite{Galatos} An algebraic structure $(L,\wedge,\vee,\odot,\rightarrow,0,1)$ of type $(2,2,2,2,0,0)$ is called a \emph{FL$_\textrm{e}$-algebra} if it satisfies the following conditions:

  (1)\ $(L,\wedge,\vee)$ is a lattice,

  (2)\ $(L,\odot,1)$ is a  commutative monoid,

  (3)\ $x\odot y\leq z$ iff $x\leq y\rightarrow z$ for any $x,y,z\in L$.
  \end{definition}

   Clearly, the class $\mathbb{FL}$ forms a variety. In what follows, we denote by $L$ the universe of a residuated lattice $(L,\wedge,\vee,\odot,\rightarrow,0,1)$. In any residuated lattice $L$, we define other operations
\begin{center}$x^0=1$ and $x^n=x^{n-1}\odot x$  for $n\geq 1$, 

${-}x=x\rightarrow 0$ and ${--} x=-(- x)$ for any $x,y\in L$.
\end{center}

  \begin{definition}\label{definietion2.2} \cite{Galatos,Jipsen} A FL$_\textrm{e}$-algebra $L$ is

  (1) \ a \emph{residuated lattice} (FL$_\textrm{ew}$-algebra) iff it satisfies the equation (INT).
  
  (2)\ \emph{distributive} iff its underlying lattice is distributive.

  (3)\ \emph{involutive} iff it satisfies the equation (INV).
  
  (4)\ an \emph{MV-algebra} iff it satisfies the equations (INV), (DIV) and (PRE).
  \end{definition}

Subsequently, we recall the results concerning the equivalence of the categories $\mathbb{RDL}$ and $\mathbb{DRDL'}$. To distinguish the objects in $\mathbb{RDL}$ and $\mathbb{DRDL'}$, we use $\mathcal{L}$ to denote the universe of the involutive residuated lattice in $\mathbb{DRDL'}$. For further exposition on this subject, we refer the reader to \cite{Galatosb,Castiglionib,Castiglionic}.

\begin{definition}\label{definition2.3.}\cite{Galatosb} An algebraic structure $(\mathcal{L},\cap,\cup,\otimes,{\sim},0,1)$ is an \emph{involutive residuated lattice} iff it satisfies the following conditions:

(1)\ $(\mathcal{L},\cap,\cup,0,1)$ is a bounded lattice,

(2)\ $(\mathcal{L},\otimes,1)$ is a commutative monoid,

(3)\ ${\sim}$ is an involution of the lattice that is a dual automorphism,

(4)\ $x\otimes y\leq z$ iff $x\leq {\sim}(y\otimes {\sim} z)$ for any $x,y,z\in \mathcal{L}$.
\end{definition}

\begin{remark}\label{remark2.4.} Note that involutive residuated lattices and dualizing residuated lattices are equivalent (see  \cite{Castiglionib}, Remark 4.1). Recalling in \cite{Castiglionib} that an algebra $(\mathcal{L},\cap,\cup,\otimes,\leadsto,0,1)$ is said to be a \emph{dualizing residuated lattice} provided it satisfies the following conditions:

(1) $(\mathcal{L},\cap,\cup,\otimes,\leadsto,0,1)$ is a residuated lattice,

(2) $0$ is a cyclic dualizing element. That is
$(x\leadsto 0)\leadsto 0=x$, for all $x\in \mathcal{L}$.
\end{remark}

An involutive residuated lattice $\mathcal{L}$ is called \emph{centered} if it contains a distinguished element, called a \emph{center}, which is a fixed point of the involution. A \emph{c-differential residuated lattice} is an involutive residuated lattice $\mathcal{L}$ with center $c$ satisfying the ``Leibniz''
 condition (see \cite{Castiglionib}, Definition 7.2):
\begin{center} 
$(x \otimes y) \cap c = ((x \cap c) \otimes y) \cup (x \otimes (y \cap c))$,
\end{center}
for all $x,y \in \mathcal{L}$. A c-differential residuated lattice is said to be \emph{distributive} if its underlying lattice is a bounded distributive lattice. \medskip

Castiglioni et al. constructed in \cite{Castiglionib} a c-differential residuated distributive lattice from a residuated distributive lattice. In what follows, we provide a summary of this construction.\medskip

Let $L\in \mathbb{RDL}$. Define $\mathbf{K}(L)$ in the following way:
\begin{center} $\mathbf{K}(L)=\{(x,y)\in L\times L^o \mid x\odot y=0\}$,
\end{center}
where $L^o$ is the dual order of $L$ (see \cite{Castiglionib}, Section 7). Define the operations in $\mathbf{K}(L)$:
\begin{center} $(x,y)\cup (z,w)=(x\vee z,y\wedge w)$,

$(x,y)\cap (z,w)=(x\wedge z,y\vee w)$,

$(x,y)\otimes(z,w)=(x\odot z,(x\rightarrow w)\wedge (z\rightarrow y))$,

$(x,y)\leadsto(z,w)=((x\rightarrow z)\wedge (w \rightarrow y), w\odot x)$.

\end{center}
Then the algebra $(\mathbf{K}(L),\cap,\cup,\otimes,\leadsto,(0,1),(1,0))$ is a dualizing residuated distributive lattice, and $(1,1)$ is a cyclic dualizing element, the induced order $\leq$ by $\cap$ in $\mathbf{K}(L)$ is given by
 \begin{center} $(x_1,y_1)\leq (x_2,y_2)$ if and only if $x_1\leq x_2$ and $y_1\geq y_2$
 \end{center}
 (see \cite{Tsinakis}, Corollary 3.5). The algebraic structure $(\mathbf{K}(L),\cap,\cup,\otimes,{\sim},(0,1),(1,0)),$ with 
\begin{center} ${\sim}(x,y)=(x,y)\leadsto (1,1)=(y,x)$
 \end{center}
is an involutive residuated lattice with $c=(0,0)$ (see  \cite{Castiglionib}, Remark 4.1), and an object of $\mathbb{DRDL}$. The assignment $\mathbf{K}$ extends to a functor $\mathbf{K}:\mathbb{RDL}\rightarrow \mathbb{DRDL}$ (see \cite{Castiglionib}, Lemma 7.1).

\medskip

Let $\mathcal{L}\in \mathbb{DRDL}$ and $\mathcal{C}(\mathcal{L})=\{x\in \mathcal{L}\mid x\geq c\}$. In  $\mathcal{C}(\mathcal{L})$ define the product
\begin{center} $x\otimes_cy=(x\otimes y)\cup c$,
\end{center}
the bottom as the constant $c$ and the other operations as those induced from $\mathcal{L}$. Then $\mathcal{C}(\mathcal{L})$ is an object of $\mathbb{DRDL}$ and $\mathcal{L}\mapsto \mathcal{C}(\mathcal{L})$ defines another functor
\begin{center} $\mathbf{C}:\mathbb{DRDL}\rightarrow\mathbb{RDL}$ ,
\end{center}
which is left adjoint to $\mathbf{K}$ (see \cite{Castiglionib}, Theorem 7.6).

The adjunction
\begin{center} $\mathbf{C}\dashv \mathbf{K}:\mathbb{RDL}\rightarrow \mathbb{DRDL}$
 \end{center}
restricts to an equivalence $\mathbf{C}\dashv \mathbf{K}:\mathbb{RDL}\rightarrow \mathbb{DRDL'}$ (see \cite{Castiglionib}, Corollary 7.8),
where $\mathbb{DRDL'}$, is the subcategory of the category $\mathbb{DRDL}$, whose objects $\mathcal{L}$ satisfy the following condition:
\begin{center} $(\mathbf{CK})$ \ ~~~For every pair of element $x,y\in \mathcal{L}$ such that $x,y\geq c$ and $x\otimes y\leq c$, there exists $z\in \mathcal{L}$ such that $z\cup c=x$ and ${\sim}z\cup c=y$.
\end{center}

Castiglioni et al. also proved in \cite{Castiglionic} that $\mathbb{DRDL'}$ forms a variety. More precisely, let $\mathcal{L}\in \mathbb{DRDL'}$. Then there is a map $\kappa:\mathcal{L}\rightarrow \mathcal{L}$ that satisfies the following two conditions:

\begin{center}

$\kappa x\cup c=c\leadsto x$~~and~~$\kappa x\cup c=x\cup c$.
\end{center}

Conversely, if $\mathcal{L} \in \mathbb{DRDL'}$ in which there exists an operator $\kappa$ that satisfies the previous equations, then $(\mathbf{CK})$ holds on $\mathcal{L}$ (see \cite{Castiglionic}, Theorem 1).\medskip

 Based on the above, we recall the results concerning the equivalence of the categories $\mathbb{MV}$ and $\mathbb{MV^\bullet}$.\medskip

Let $\mathbb{MV^\bullet}$ be the full subcategory of $\mathbb{DRDL'}$ whose object $\mathcal{L}$ satisfying 
\begin{center}
$\text{(Inv$^\bullet$)}~{\sim}\kappa x=\kappa{\sim}\kappa x$,
 
$\text{(Lin$^\bullet$)}~(x\leadsto y)\cup (y\leadsto x)\geq c$, 

$\text{(QHey$^\bullet$)}~(x\otimes c)\otimes (x\leadsto(y\cup c))=(x\cap y)\otimes c$,
\end{center}
 for any $x,y\in \mathcal{L}$  [see \cite{Castiglionic}, Theorem 13]. The category $\mathbb{MV^\bullet}$, whose objects are MV$^\bullet$-algebras, is the image of the category $\mathbb{MV}$, whose objects are MV-algebras, under the categorical equivalence $\mathbf{K}^\bullet$. More precisely, for $\mathcal{L}\in
\mathbb{MV^\bullet}$, we have
\begin{center} $\kappa(\mathcal{L})=\{x\in \mathcal{L}\mid \kappa(x)=x\}$.
\end{center}
The assignment $\mathcal{L}\mapsto \kappa(\mathcal{L})$ extends to a functor $\mathscr{K}:\mathbb{MV^\bullet}\rightarrow \mathbb{MV}$ [see \cite{Castiglionic}, Theorem 10]. If $g:\mathcal{L}_1\rightarrow \mathcal{L}_2$ is a morphism in $\mathbb{MV^\bullet}$, then $\mathscr{K}(g):\kappa(\mathcal{L}_1)\rightarrow \kappa(\mathcal{L}_2)$ is the morphism in $\mathbb{MV}$ given by the restriction of $g$ to $\kappa(\mathcal{L}_1)$. Conversely, for $L\in \mathbb{MV}$, the assignment $L\mapsto \mathbf{K}^\bullet(L)$ extends to a functor $\mathbf{K}^\bullet:\mathbb{MV}\rightarrow \mathbb{MV^\bullet}$. If $L_1\rightarrow L_2$ is a morphism in $\mathbb{MV}$, then
$\mathbf{K}^\bullet(f):L_1\rightarrow L_2$ is the morphism in $\mathbb{MV^\bullet}$ given by 
\begin{center}$(\mathbf{K}^\bullet(f))(x,y)=(f(x),f(y))$.
\end{center} The adjunction $\mathscr{K} \dashv \mathbf{K}^\bullet:\mathbb{MV^\bullet}\rightarrow \mathbb{MV}$ yields an equivalence [see \cite{Castiglionic}, Corollary~15]. 
This equivalence is realized explicitly via natural isomorphisms: for any $L\in \mathbb{MV}$, 
the isomorphism 
\[
\alpha:L\longrightarrow \kappa(\mathbf{K}^\bullet(L)), \quad \alpha(x)=(x,{\sim} x),
\] 
and for any $\mathcal{L}\in \mathbb{MV^\bullet}$, the isomorphism 
\[
\beta:\mathcal{L}\longrightarrow \mathbf{K}^\bullet(\kappa(\mathcal{L})), \quad 
\beta(x)=(\lambda(x),\lambda({\sim} x)), \text{ where } \lambda(x)={\sim} \kappa({\sim} x)
\] 
[see \cite{Castiglionic}, Theorem~11].

\section{Monadic residuated lattices}

 The variety of monadic FL$_\textrm{e}$-algebras was introduced by Tuyt in his doctoral dissertation, as an algebraic semantic for one-variable fragments of first-order substructural logics \cite{Tuyt}.

 \begin{definition}\label{definition3.1} \cite{Tuyt} A \emph{monadic FL$_\textrm{e}$-algebra} is a triple $(L,\square,\lozenge)$, where $L$ is a FL$_\textrm{e}$-algebra and $\square,\lozenge:L\rightarrow L$ are two unary operators on $L$ satisfying the conditions:

 (M1)\ $\square x\leq x$,

 (M2)\ $\square\lozenge x=\lozenge x$,

 (M3)\ $\square(x\wedge y)=\square x\wedge \square y$,

 (M4)\ $\square 0=0$,

 (M5)\ $\square 1=1$,

 (M6)\ $\square (x\rightarrow\square y)=\lozenge x\rightarrow\square y$,

 (M7)\ $\square (\square x\rightarrow y)=\square x\rightarrow\square y$,\\
 for any $x,y\in L$.
 \end{definition}

The class $\mathbb{MFL}_\textrm{e}$ of monadic FL$_\textrm{e}$-algebras forms a variety.

 Here we give an equivalent characterization of monadic residuated lattices.

 \begin{theorem}\label{theorem3.2}Let $L$ be a residuated lattice and $\square,\lozenge$ be two unary operators on the residuated lattice $L$. Then the following statements are equivalent:

 (1)\ $(L,\square,\lozenge)$ is a monadic residuated lattice,

 (2)\ $\square$ and $\lozenge$ satisfy the identities (M1)-(M3) and (M6)-(M7).
 \end{theorem}
 \begin{proof} $(1)\Rightarrow (2)$ It can be directly follows from Definition \ref{definition3.1}.

 $(2)\Rightarrow(1)$ Here we only need to show that (M4) and (M5) hold.

 (M4) Taking $x=0$ in (M1), we have
 \begin{center} $\square 0\leq 0$,
 \end{center} which implies $\square 0=0$.

 (M5) Taking $x=y=1$ in (M1) we get $\square1\to 1=1$, thus by (M7), we have
 \begin{center} $\square1=\square(\square 1\rightarrow 1)=\square 1\rightarrow \square 1=1$,
 \end{center} which implies $\square 1=1$.
 \end{proof}

Now we introduce the variety of monadic residuated lattices in the sense of Tuyt, 
and study some of their some subvarieties, for instance, monadic involutive
residuated lattices and UMV-algebras.

\begin{definition}\label{definition3.3} A monadic residuated lattice is a triple $(L,\square,\lozenge)$ that satisfies the identities: (M1)-(M3) and (M6)-(M7).
\end{definition}

The unary operators $\square,\lozenge:L\rightarrow L$ are called a universal and an existential quantifier on a residuated lattice $L$, respectively. The class $\mathbb{MRL}$ of monadic residuated lattices clearly forms a variety.

 \begin{proposition}\label{proposition3.4} In any monadic residuated lattice $(L,\square,\lozenge)$, the following identities hold:

  (1)\ $x\leq \lozenge x$,

  (2)\ $x\leq y$ implies $\square x\leq \square y$,

  (3)\ $x\leq y$ implies $\lozenge x\leq \lozenge y$,

  (4)\ $\square x=\lozenge\square x$,

  (5)\ $\square (\lozenge x\rightarrow y)=\lozenge x\rightarrow\square y$,

  (6)\ $\square (x\rightarrow \lozenge y)=\lozenge x\rightarrow\lozenge y$,

  (7)\ $\square\square x=\square x$,

  (8)\ $\lozenge\lozenge x=\lozenge x$,

  (9)\ $\lozenge 1=1$,

  (10)\ $\lozenge 0=0$,

  (11)\ $\lozenge(\lozenge x\odot \lozenge y)=\lozenge x\odot \lozenge y$,

  (12)\ $\square(\square x\odot \square y)=\square x\odot \square y$,

  (13)\ $\lozenge(x\odot \lozenge y)=\lozenge x\odot \lozenge y$,

  (14)\ $\square(\lozenge x\vee \lozenge y)=\lozenge x\vee \lozenge y$,

  (15)\ $\lozenge(x\vee y)=\lozenge x\vee \lozenge y$,

  (16)\ $\lozenge(x\rightarrow \lozenge y)\leq \lozenge x\rightarrow\lozenge y$,
  
  (17)\ $ \square x\odot \square y\leq \square(x\odot y)$,
  
  (18)\ $\lozenge(x\odot \square y)=\lozenge x\odot \square y$,
  
  (19)\ $\lozenge(x\odot y)\geq \lozenge x\odot \square y$,
  
  (20)\ $\square(\square x\rightarrow \square y)=\square x\rightarrow \square y$,

  (21)\ $\square{-}\square x={-}\square x$,

  (22)\ $\square{-}x={-}\lozenge x$,
  
  (23) $\lozenge \square x \leq x$,

  (24) $x\leq \square \lozenge x$,

  (25) $ \lozenge x \leq y$ if and only if $x\leq \square y$.
 \end{proposition}
 \begin{proof} Noticing here that every monadic residuated lattice is a monadic FL$_\textrm{e}$-algebra, and so the proof of (1)-(16) are similar to that of Lemma 2.2(L7)-(L23) in \cite{Tuyt}, respectively.

 (18)\ By (4) and (13), we have
 \begin{center} $\lozenge(x\odot \square y)=\lozenge(x\odot \lozenge \square y)=\lozenge x\odot \lozenge \square y=\lozenge x\odot \square y$,
 \end{center} which implies $\lozenge(x\odot \square y)= \lozenge x\odot \square y$.
 
 (19)\ By Definition \ref{definition3.1} (M1) and the monotonicity of $\odot$, $x\odot y\geq x\odot\square y$, then by (3) and (18) we have
  
 \begin{center} $\lozenge(x\odot y)\geq \lozenge(x\odot \square y)=\lozenge x\odot \square y$,
 \end{center} which implies $\lozenge(x\odot y)\geq \lozenge x\odot \square y$.

 (20)\  By Definition \ref{definition3.1}(M7) and (7), we have
 \begin{center} $\square (\square x\rightarrow \square y)=\square x\rightarrow \square\square y=\square x\rightarrow\square y$,
 \end{center} which implies $\square(\square x\rightarrow \square y)=\square x\rightarrow \square y$.

 (21)\ Taking $y=0$ in Definition \ref{definition3.1}(M7) and from Definition \ref{definition3.1}(M4), we have
 \begin{center} $\square{-}\square x=\square (\square x\rightarrow 0)=\square x\rightarrow\square 0={-}\square x$,
 \end{center}which implies $\square{-}\square x={-}\square x$.

 (22)\ Taking $y=0$ in (6) and from (10), we have
 \begin{center} $\square{-}x=\square (x\rightarrow \lozenge 0)=\lozenge x\rightarrow\lozenge 0={-}\lozenge x$,
 \end{center} which implies $\square{-}x={-}\lozenge x$.

 (23)
 This follows from (4) and Definition \ref{definition3.1} (M1).

(24)
 As an application of \ref{definition3.1} (M5), (6), (1) and residuation properties we have 
 \[
 1=\square1=\square(\lozenge x \to \lozenge x)=\lozenge x\to \square\lozenge x\leq x\to \square\lozenge x,
 \]
 which implies $x\leq \square \lozenge x$.

(25)
 This is basically a restatement of (23) and (24).
 \end{proof}

In order to facilitate the definition of monadic c-differential residuated lattice in Section 4, we introduce the notion of monadic involutive residuated lattices in advance.

\begin{definition}\label{definition3.5} An algebra $(L,\wedge,\vee,\odot,\rightarrow,\square,0,1)$ is said to be a \emph{ monadic involutive residuated lattice} if $(L,\wedge,\vee,\odot,\rightarrow,0,1)$ is an involutive residuated lattice and in addition $\square$ satisfies the identities: (M1),(M3) and (M7).
\end{definition}

The class $\mathbb{MIRL}$ of monadic involutive residuated lattice forms a variety.

\begin{theorem}\label{theorem3.6} The subvariety of the variety $\mathbb{MRL}$ determined by the
 \begin{center} ${--}x=x$,
 \end{center}
 is term-equivalent to the variety $\mathbb{MIRL}$.
\end{theorem}
\begin{proof} Let  $(L,\square,\lozenge)$ be a monadic residuated lattice that satisfies ${--}x=x$. Then it directly follows from Definition \ref{definition3.1}(M1), (M3) and (M7) that  $(L,\square)$ is a monadic involutive residuated lattice.


Conversely, let $(L,\square)$ be a monadic involutive residuated lattice and define $\lozenge x := -\square - x$. Then we show that $(L,\square,\lozenge)$ is a monadic residuated lattice. Notice that (M1), (M3) and (M7) are directly follows from Definition \ref{definition3.5}.

(M2):\ By Definition \ref{definition3.1} (M1) and (M7), we have
\begin{center} $\square{-}\square{-}x=\square(\square{-}x\rightarrow 0)=\square{-x}\rightarrow \square0=\square{-x}\rightarrow 0={-}\square{-x}$,
\end{center} which implies
\begin{center} $\square \lozenge x=\lozenge x$.
\end{center}

(M6):\ In order to prove this, we first to show that
 \begin{center} $\square({-}x\rightarrow \square y)={-}\square x\rightarrow\square y$.
 \end{center}
Indeed, from Definition \ref{definition3.1} (M1) and (M7), we have
\begin{center} $\square{-}\square x={-}\square x$.
\end{center} It is worth noticing that by Definition \ref{definition3.1} (M7) and above equation again, we also have
\begin{eqnarray*} \square(x\rightarrow \square y)&=&\square({--}x\rightarrow \square y)\\&=&\square({-}\square y\rightarrow {-}x)\\&=&{-}\square y\rightarrow \square {-}x\\&=&\lozenge x\rightarrow \square y,
\end{eqnarray*} which implies

\begin{center} $\square(x\rightarrow \square y)=\lozenge x\rightarrow \square y$.
\end{center}
So $(L,\square,\lozenge)$ is a monadic residuated lattice.
\end{proof}

We recall that (\cite{Rodriguez}) an algebra $(L,\rightarrow,-,1)$
of type $(2,1,0)$ is a \emph{Wajsberg algebra} provided:

(W1)\ $1 \rightarrow x = x$,

(W2)\ $(x \rightarrow y) \rightarrow \bigl((y \rightarrow z) \rightarrow (x \rightarrow z)\bigr) = 1,$

(W3)\ $(x \rightarrow y) \rightarrow y = (y \rightarrow x) \rightarrow x,$

(W4)\ $(-y \rightarrow -x) \rightarrow (x \rightarrow y) = 1.$

Lattanzi \cite{Lattanzi,Lattanzia}
 introduced the notion of U-operators on Wajsberg algebras. The resulting class of algebras is called UW-algebras, and provide an algebraic counterpart of a fragment of the monadic infinite-valued {\L}ukasiewicz functional propositional calculus. More precisely, a UW-algebra is a pair $(L,\forall)$,  where $L$ is a Wajsberg algebra and $\forall:L\rightarrow L$ is a unary operator satisfying
 \begin{center} (U1) $\forall x\rightarrow x=1$~~~and~~~(U2) $\forall(\forall x\rightarrow y)=\forall x\rightarrow \forall y$.
 \end{center} 
It is well known that the categories of MV-algebras and bounded Wajsberg algebras are equivalent, so the definition of UMV-algebras can be stated analogously to that of UW-algebras. 
The class $\mathbb{UMV}$ forms a variety. 

 Moreover, monadic involutive residuated lattices are actually natural generalization of UMV-algebras. 
 
 \begin{remark}\label{remark3.7} 
 If $(L,\square)$ is a monadic involutive residuated lattice, then the following are equivalent:
 
 (1) $L$ is an MV-algebra,
 
 (2) $(L,\square)$ is a UMV-algebra.
\end{remark}
\begin{proof} It is straightforward that $(1)\Leftrightarrow (2)$. More precisely,

 $(1)\Rightarrow (2)$ follows from the fact that (U1) and (U2) hold because $(L,\square)$ satisfies (M1) and (M7). 
 
 $(2)\Rightarrow (1)$ is immediate.
\end{proof}

As a consequence of Remark \ref{remark3.7}, the following corollary is obtained, which will be used to discuss the relationship between Figures 1 and 2 in Section 4.

\begin{corollary}\label{corollary3.8} (1) The subvariety of the variety $\mathbb{MIRL}$ determined by the equations (DIV) and (PRE)
 is term-equivalent to the variety $\mathbb{UMV}$.
 
 (2) The subvariety of the variety $\mathbb{MRL}$ determined by the equations (INV), (DIV) and (PRE)
 is term-equivalent to the variety $\mathbb{UMV}$.
\end{corollary}

\section{Kalman structures derived from monadic residuated distributive lattices}\label{section4.}

\subsection{Monadic c-differential residuated lattices}

In this section, we introduce the notion of monadic c-differential residuated lattices and prove that the classes of monadic c-differential residuated distributive lattices satisfying $\mathbf{CK}$ and monadic residuated distributive lattices are in one to one correspondence. Indeed, we also construct two universal quantifiers on the c-differential residuated distributive lattices from a monadic residuated distributive lattice, which will play an important role in proving the main results of this section.

Let $(L,\square,\lozenge)$ be a monadic residuated lattice. Then we define a unary operator

\begin{center} $\square_K:\mathbf{K}(L)\rightarrow \mathbf{K}(L)$
\end{center}
such that
\begin{center} $\square_K(x,y)=(\square x, \lozenge y)$,
\end{center}
for all $x,y\in \mathbf{K}(L)$.

\begin{theorem}\label{theorem4.1} Let $(L,\square,\lozenge)$ be a monadic residuated distributive lattice. Then for any $(x,y),(x_1,y_1),(x_2,y_2)\in \mathbf{K}(L)$, we have

(1) $\square_K$ is well defined, i.e., $\square_K(x,y)\in \mathbf{K}(L)$,

(2) $\square_K(0,1)=(0,1)$,

(3) $\square_K(c)=c$,

(4) $\square_K(1,0)=(1,0)$,

(5) $\square_K(x,y)\leadsto (x,y)=(1,0)$,

(6) $\square_K((x_1,y_1)\leadsto (x_2,y_2))\leq \square_K(x_1,y_1)\leadsto \square_K(x_2,y_2)$,

(7) $\square_K(x_1,y_1)\cap (x_2,y_2))=\square_K(x_1,y_1)\cap \square_K(x_2,y_2)$,

(8) $\square_K(\square_K(x_1,y_1)\leadsto (x_1,y_2))=\square_K(x_1,y_1)\leadsto \square_K(x_2,y_2)$,

(9) $\square_K(x_1,y_1)\otimes \square_K(x_2,y_2)\leq \square_K((x_1,y_1)\otimes (x_2,y_2))$,

(10) $(x_1,y_1)\leq (x_2,y_2)$ implies $\square_K(x_1,y_1)\leq \square_K(x_2,y_2)$,

(11) $\square_K(\square_K(x_1,y_1)\leadsto \square_K(x_2,y_2))=\square_K(x_1,y_1)\leadsto \square_K(x_2,y_2)$,

(12) $\square_K(\neg\square_K(\neg(x,y)))=\neg\square_K(\neg(x,y))$,

(13) $\square_K(\square_K(x,y))=\square_K(x,y)$,

(14) $\square_K((x,y)\cap c)=\square_K(x,y)\cap c$,

(15) $\square_K((x,y)\cup c)=\square_K(x,y)\cup c$,

(16) $\square_K((x,y)\otimes c)=\square_K(x,y)\otimes c$.
\end{theorem}

\begin{proof}(1):\ If $(x,y)\in \mathbf{K}(L)$, then $x\leq {-}y$. By Propositions \ref{proposition3.4}(2) and \ref{proposition3.4}(19), we have 
\begin{center} $\square x\leq \square {-}y={-}\lozenge y$,
\end{center} and hence $\square x\odot \lozenge y=0$, that is $(\square x,\lozenge y)\in \mathbf{K}(L)$, which implies that $\square_K$ is well defined.

(2):\ From Definition \ref{definition3.1}(M4) and Proposition \ref{proposition3.4}(9), we have

\begin{center} $\square_K(0,1)=(\square 0,\lozenge 1)=(0,1)$.
\end{center}

(3):\ From Definition \ref{definition3.1}(M4)  and Proposition \ref{proposition3.4}(10), we have

\begin{center} $\square_K(c)=\square_K(0,0)=(\square 0,\lozenge 0)=(0,0)=c$.
\end{center}

(4):\ From Definition \ref{definition3.1}(M5) and Proposition \ref{proposition3.4}(10), we have 
\begin{center} $\square_K(1,0)=(\square 1,\lozenge 0)=(1,0)$.
\end{center}

(5):\ If $(x,y)\in \mathbf{K}(L)$, then by Proposition \ref{proposition3.4} (1) we have 
\begin{center}
$\square x\odot y\leq \square x\odot \lozenge y=0$,
\end{center}
and hence
\begin{center} $\lozenge x\odot y=0$, 
\end{center}
which implies
\begin{eqnarray*}
               \square_K(x,y)\leadsto (x,y)&=& (\square x, \lozenge y)\leadsto(x,y)\\&=& ((\square x\rightarrow x)\wedge(y\rightarrow \lozenge y),\square x\odot y)\\&=&(1,0).
                \end{eqnarray*}

(6):\ Let $(x_1,y_1),(x_2,y_2)\in \mathbf{K}(L)$. Then by Proposition \ref{proposition3.4}(17), we have
\begin{eqnarray*}
               \square_K((x_1,y_1)\leadsto(x_2,y_2)) &=& \square_K((x_1\rightarrow x_2)\wedge(y_2\rightarrow y_1), y_2\odot x_1)\\ &=& (\square((x_1\rightarrow x_2)\wedge(y_2\rightarrow y_1)),\lozenge (y_2\odot x_1))\\ &\leq & ((\square x_1\rightarrow \square x_2)\wedge (\lozenge y_2\rightarrow \lozenge y_1), \lozenge y_2\odot \square x_1).
               \end{eqnarray*}
On the other hand, we have
\begin{eqnarray*}
               \square_K(x_1,y_1)\leadsto \square_K(x_2,y_2)&=& (\square x_1,\lozenge y_1)\leadsto (\square x_2,\lozenge y_2)\\ &=& ((\square x_1\rightarrow \square x_2)\wedge (\lozenge y_2\rightarrow \lozenge y_1), \lozenge y_2\odot \square x_1).
               \end{eqnarray*}
                              
(7):\ Let $(x_1,y_1),(x_2,y_2)\in \mathbf{K}(L)$. Then by Definition \ref{definition3.1}(M3) and Proposition \ref{proposition3.4}(15), we have
\begin{eqnarray*}
               \square_K((x_1,y_1)\sqcap(x_2,y_2))&=& \square_K(x_1\wedge x_2,y_1\vee y_2)\\ &=& (\square(x_1\wedge x_2),\lozenge(y_1\vee y_2))\\ &=& (\square x_1\wedge \square x_2,\lozenge y_1\vee \lozenge y_2)\\ &=& (\square x_1,\lozenge y_1)\sqcap (\square x_2, \lozenge y_2)\\ &=& \square_K(x_1,y_1)\sqcap \square_K(x_2,y_2).
               \end{eqnarray*}

(8):\ Let $(x_1,y_1),(z,w)\in \mathbf{K}(L)$. Then by Proposition \ref{proposition3.4}(18), we have
\begin{eqnarray*}
               \square_K(\square_K(x_1,y_1)\rightsquigarrow(x_2,w))&=& \square_K((\square x_1,\lozenge y_1)\rightsquigarrow(x_2,w))\\ &=& \square_K((\square x_1\rightarrow x_2)\wedge(w\rightarrow\lozenge y_1),\square x_1\odot w)\\ &=& (\square((\square x_1\rightarrow x_2)\wedge(w\rightarrow\lozenge y_1)),\lozenge(\square x_1\odot w))\\ &=& ((\square x_1\rightarrow\square x_2)\wedge (\square w\rightarrow\lozenge y_1),\square x_1\odot \lozenge w).
               \end{eqnarray*}
On the other hand, we have
\begin{eqnarray*}
               \square_K(x_1,y_1)\rightsquigarrow \square_K(x_2,y_2)&=& (\square x_1,\lozenge y_1)\rightsquigarrow (\square x_2,\lozenge y_2)\\ &=& ((\square x_1\rightarrow \square x_2)\wedge (\lozenge y_2\rightarrow\lozenge y_1),\square x_1\odot\lozenge y_2).
               \end{eqnarray*}
               
(9):\  By (6) and residuation, we have
\begin{center} $\square_K((x_1,y_1)\leadsto (x_2,y_2))\otimes \square_K(x_1,y_1)\leq \square_K(x_2,y_2)$.
\end{center}
Taking $(x_2,y_2)=(x_2,y_2)\otimes (x_1,y_1)$ in the above equation, we have
\begin{eqnarray*} \square_K(x_1,y_1)\otimes \square_K(x_2,y_2)&=&\square_K(x_2,y_2)\otimes \square_K(x_1,y_1)\\&\leq&  \square_K((x_1,y_1)\leadsto ((x_2,y_2)\otimes (x_1,y_1)))\otimes \square_K(x_1,y_1)\\&\leq& \square_K((x_1,y_1)\otimes (x_2,y_2)),
\end{eqnarray*} which implies
\begin{center} $\square_K(x_1,y_1)\otimes \square_K(x_2,y_2)\leq \square_K((x_1,y_1)\otimes (x_2,y_2))$.
\end{center}

(10):\ If $(x_1,y_1)\leq (x_2,y_2)$, then by (4) and (6), we have
\begin{center}$\square_K((x_1,y_1)\leadsto (x_2,y_2))=\square_K(1,0)=(1,0)\leq \square_K(x_1,y_1)\leadsto \square_K(x_2,y_2)$,
\end{center} which implies $\square_K(x_1,y_1)\leq \square_K(x_2,y_2)$.

(11):\ Let $(x_1,y_1),(x_2,y_2)\in \mathbf{K}(L)$. Then by Proposition \ref{proposition3.4}(18), we have
\begin{eqnarray*}
                \square_K( \square_K(x_1,y_1)\leadsto \square_K(x_2,y_2))&=&  \square_K((\square x_1,\lozenge y_1)\leadsto (\square x_2,\lozenge y_2)\\ &=& ((\square x_1\rightarrow \square x_2)\wedge(\lozenge y_2\rightarrow \lozenge y_1),\square x_1\odot \lozenge y_2)\\ &=& (\square(\square x_1\rightarrow \square x_2)\wedge(\lozenge y_2\rightarrow \lozenge y_1)),\lozenge (\square x_1\odot \lozenge y_2)\\ &=& ((\square x_1 \rightarrow \square x_2)\wedge (\lozenge y_2\rightarrow \lozenge y_1),\square x_1\odot \lozenge y_2).
               \end{eqnarray*}
On the other hand, we have
\begin{eqnarray*}
                \square_K(x_1,y_1)\leadsto  \square_K(x_2,y_2)&=& (\square x_1,\lozenge y_1)\leadsto (\square x_2,\lozenge y_2)\\ &=& ((\square x_1 \rightarrow \square x_2)\wedge (\lozenge y_2\rightarrow \lozenge y_1),\square x_1\odot \lozenge y_2),
               \end{eqnarray*} which implies 
\begin{center} $\square_K( \square_K(x_1,y_1)\leadsto \square_K(x_2,y_2))=\square_K(x_1,y_1)\leadsto  \square_K(x_2,y_2)$.
\end{center}

(12):\ Let $(x,y)\in \mathbf{K}(L)$. Then by Definition \ref{definition3.1}(M2), we have
\begin{eqnarray*}
               \square_K(\neg\square_K(\neg(x,y)))&=& \square_K(\neg\square_K(y,x))\\&=& \square_K(\square_K(y,x)\rightsquigarrow(1,1))\\&=&\square_K((\square y,\lozenge x)\rightsquigarrow (1,1))\\&=&\square_K(((\square y\rightarrow 1)\wedge (1\rightarrow\lozenge x),\lozenge x))\\&=&\square_K((\lozenge x,\lozenge x))\\&=&(\square\lozenge x, \lozenge x)\\&=&(\lozenge x, \lozenge x).
               \end{eqnarray*}
On the other hand, we have
\begin{eqnarray*}
               \neg\square_K(\neg(x,y))&=& \neg\square_K(y,x)\\&=& \square_K(y,x)\rightsquigarrow(1,1)\\&=&(\square y,\lozenge x)\rightsquigarrow (1,1)\\&=&((\square y\rightarrow 1)\wedge (1\rightarrow\lozenge x),\lozenge x)\\&=&(\lozenge x,\lozenge x).
               \end{eqnarray*}
               
(13):\ Let $(x,y)\in \mathbf{K}(L)$. Then by Propositions \ref{proposition3.4}(7) and (8), we have
\begin{eqnarray*}
\square_K(\square_K(x,y))&=&\square_K(\square x, \lozenge y)\\&=& (\square\square x,\lozenge \lozenge y)\\&=&(\square x,\lozenge y)\\&=& \square_K(x,y).
\end{eqnarray*}
               
(14):\ Let $(x,y)\in \mathbf{K}(L)$. Then by Definition \ref{definition3.1}(M4), we have
\begin{eqnarray*}
\square_K((x,y)\cap c)&=&\square_K((x,y)\cap(0,0))\\&=& \square_K(0,y)\\&=&(\square 0, \lozenge y)\\&=& (0,\lozenge y)\\&=&\square_K(x,y)\cap c.
\end{eqnarray*}

(15):\ Let $(x,y)\in \mathbf{K}(L)$. Then by Proposition \ref{proposition3.4}(10), we have
\begin{eqnarray*}
\square_K((x,y)\cup c)&=&\square_K((x,y)\cup(0,0))\\&=&\square_K(x,0)\\&=&(\square x, \lozenge 0)\\&=&(\square x,0)\\&=&\square_K(x,y)\cup c.
\end{eqnarray*}

(16):\ Let $(x,y)\in \mathbf{K}(L)$. Then by Definition \ref{definition3.1}(M4), we have
\begin{eqnarray*}
\square_K((x,y)\otimes c)&=&\square_K((x,y)\otimes(0,0))\\&=&\square_K(0,y)\\&=&(0,\lozenge y)\\&=&(\square 0, \lozenge y)\otimes c\\&=&\square_K((x,y)\otimes c.
\end{eqnarray*}
\end{proof}

\begin{remark}\label{remark4.2} It should be pointed out here that if $L$ is a residuated distributive lattice, then $\mathbf{K}(L)$ is a special class of an involutive residuated lattice, with Theorems \ref{theorem4.1}(5),(7) and (8), which forms a monadic involutive residuated lattice by Definition \ref{definition3.5}.
\end{remark}

Inspired by Definition \ref{definition3.5} and Theorem \ref{theorem4.1}(5),(7) and (8), and combined by the compatibility of universal quantifiers and the center element $c$ presented by Theorem \ref{theorem4.1}(14)-(16), we introduce the notion of center universal quantifier on c-differential residuated lattices, and the resulting class of algebras will be called monadic c-differential residuated lattices.

\begin{definition}\label{definition4.3} A \emph{monadic c-differential residuated lattice} is a pair 
\begin{center} $(\mathcal{L},\cap,\cup,\otimes,\neg, \square_c,0,c,1,)$
\end{center} where
$(\mathcal{L},\cap,\cup,\otimes,\neg,0,c,1)$ is a c-differential residuated lattice and in addition $\square_c$ satisfies the identities: 

($\square_c1$)\ $\square_c c=c$,

($\square_c2$)\ $\square_cx\leadsto x=1$,

($\square_c3$)\ $\square_c(x\cap y)=\square_c x\cap \square_c y$,

($\square_c 4$)\ $\square_c(\square_c x\leadsto y)=\square_c x\leadsto \square_c y$,

($\square_c 5$)\ $\square_c(x\cap c)=\square_cx\cap c$,

($\square_c 6$)\ $\square_c(x\cup c)=\square_c x\cup c$,

($\square_c 7$)\ $\square_c(x\otimes c)=\square_c x\otimes c$.

\end{definition}

The $\square_c:\mathcal{L}\rightarrow \mathcal{L}$ is called a center universal quantifier on a c-differential residuated lattice $\mathcal{L}$.

The class $\mathbb{MDRDL}$ of monadic c-differential residuated lattices forms a variety.  

\begin{remark}\label{remark4.4.} (1) Taking $\mathcal{L}=\mathbf{K}(L)$ when $L$ is a residuated distributive lattice, then by Theorem \ref{theorem4.1} (5),(7),(8),(14)-(16), we have that $\square_K$ is a center universal quantifier on a c-differential residuated lattice on $\mathbf{K}(L)$, and it can be seen as the image of the universal and existential quantifiers $\forall$ and $\exists$, under the equivalence kalman functor $\mathbf{K}$. 

(2) Note that c-differential residuated lattices are the special kinds of involutive residuated lattice, and hence the notion of universal quantifiers is defined in Definition \ref{definition3.5}  can be also directly applied from involutive residuated lattices to c-differential residuated lattices. It is remarked here that universal quantifiers and center universal quantifiers 
 on c-differential residuated lattices are not the same. For example, let $L$ be a residuated distributive lattice and $\mathbf{K}(L)$ be the corresponding Kalman structure, 
\begin{center}$\square_{\neg}:\mathbf{K}(L)\rightarrow \mathbf{K}(L)$ 
\end{center} be a unary operator defined by 
\begin{center} $\square_{\neg}(x,y)=(x,\neg x)$,
\end{center} for any $(x,y)\in \mathbf{K}(L)$. Then $\square_{\neg}$ is a universal quantifier of Definition \ref{definition3.5} on $\mathbf{K}(L)$, but it is not a center universal quantifier of Definition \ref{definition4.3}.

 More precisely, if $(x,y)\in \mathbf{K}(L)$, then $x\odot \neg x=0$, which means that $\square_{\neg}$ is well defined.

(M1)\ If $(x,y)\in \mathbf{K}(L)$, then $ x\odot y=0$, and hence
\begin{eqnarray*}
               \square_{\neg}(x,y)\leadsto (x,y)&=& (x,{\neg}x)\leadsto (x,y)\\&=&((x\rightarrow x)\wedge(y\rightarrow {-}x),x\odot y)\\&=&(1,0).
               \end{eqnarray*}

(M3) \ Let $(x_1,y_1),(x_2,y_2)\in \mathbf{K}(L)$. Then
\begin{eqnarray*}
               \square_{\neg}((x_1,y_1)\cap(x_2,y_2))&=& \square_{\neg}(x_1\wedge x_2,y_1\vee y_2)\\ &=& (x_1\wedge x_2,{-}(x_1\wedge x_2))\\ &=& (x_1\wedge x_2,{\neg} x_1\vee {\neg} x_2 )\\ &=& (x_1,{\neg}x_1)\cap (x_2,{\neg}x_2)\\ &=& \square_{\neg}(x_1,y_1)\cap \square_{\neg}(x_2,y_2).
               \end{eqnarray*}

(M7)\ Let $(x_1,y_1),(x_2,y_2)\in \mathbf{K}(L)$. First noticing that
 \begin{center} $x_1\odot (x_1\rightarrow x_2)\odot y_2\leq x_2\odot y_2=0$,
  \end{center}
 which directly implies $x_1\rightarrow x_2\leq y_2\rightarrow {-}x_1$ by Definition 2.1(3). Then we have
\begin{eqnarray*}
               \square_{\neg}(\square_{\neg}(x_1,y_1)\leadsto(x_2,y_2))&=& \square_{\neg}((x_1,{-}x_1)\leadsto (x_2,y_2))\\ &=& \square_{\neg}((x_1\rightarrow x_2)\wedge(y_2\rightarrow {-}x_1),x_1\odot y_2)\\ &=& \square_{\neg}(x_1\rightarrow x_2),x_1\odot y_2)\\ &=& (x_1\rightarrow x_2,{-}(x_1\rightarrow x_2))\\ &=& (x_1\rightarrow x_2,x_1\odot {-}x_2).
               \end{eqnarray*}
On the other hand, we have
\begin{eqnarray*}
               \square_{\neg}(x_1,y_1)\leadsto \square_{\neg}(x_2,y_2)&=& (x_1,{-}x_1)\leadsto (x_2,{-}x_2)\\ &=& ((x_1\rightarrow x_2)\wedge ({-}x_2\rightarrow {-}x_1), x_1\odot {-}x_2)\\ &=& (x_1\rightarrow x_2,x_1\odot {-}x_2).
               \end{eqnarray*}            
Then by Definition \ref{definition3.5} we have that $\square_{\neg}$ is a universal quantifier on a c-differential residuated lattice $\mathbf{K}(L)$.
However, $\square_{\neg}$ is not a center universal quantifier of Definition \ref{definition4.3} on a center universal quantifier, since  Definition \ref{definition4.3} ($\square_c1$) does not hold generally,
\begin{center} $\square_{\neg}(c)=(0,0)=\square_{\neg}(0,\neg 0)=(0,1)\neq (0,0)=c$.
\end{center}
\end{remark}

\begin{proposition}\label{proposition4.5} In any monadic c-differential residuated lattice $(\mathcal{L},\cap,\cup,\otimes,\neg, \square_c,0,c,1)$, the following properties hold:

($\square_c 8$)\ $\square_c(1)=1$,

($\square_c 9$)\ $\square_c(0)=0$,

($\square_c 10$)\ $\square_c\square_c x=\square_c x$,

($\square_c 11$)\ $x\leq y$ implies $\square_c x\leq \square_c y$,

($\square_c 12$)\ $\square_c x\otimes \square_c y\leq \square_c(x\otimes y)$,

($\square_c 13$)\ $\square_c\neg \square_c x=\neg \square_c x$,

($\square_c 14$)\ $\square_c(\square_c x\leadsto \square_c y)=\square_c x\leadsto \square_c y$.
\end{proposition}

\begin{proof} The proofs of ($\square_c 8$)-($\square_c 14$) can be directly deduced from Definition \ref{definition3.1} and Proposition \ref{proposition3.4}. 
\end{proof}

\begin{theorem}\label{theorem4.6} Let $L$ be a residuated distributive lattice. Then
\[
\varphi : (L,\square ,\lozenge) \; \longmapsto \; (\mathbf{K}(L),\square_K)
\]
defines a mapping from the class of  monadic expansions of $L$ to the class of monadic c-differential residuated distributive lattices expansions of $\mathbf{K(L)}$, which satisfies the following two properties:

   (1)~~If $(L, \square_1, \lozenge_1)$ and $(L, \square_2, \lozenge_2)$ are two monadic expansions of $L$, then
    \[
    \mathbf{K}(L, \square_1, \lozenge_1) \cong \mathbf{K}(L, \square_2, \lozenge_2)
    \]
    as monadic c-differential residuated distributive lattices expansions of $\mathbf{K(L)}$, which shows that the image under $\mathbf{K}$ of any monadic expansion of $L$ is essentially unique.

   (2)~~If $(\mathbf{K}(L), \square_K)$ is a monadic c-differential residuated distributive lattice, then there exist the binary operators $\square, \lozenge$ on $L$ such that $(L, \square, \lozenge)$ is a monadic residuated distributive lattice and
    \[
    \varphi(L, \square, \lozenge) \cong (\mathbf{K}(L), \square_K).
    \]

Consequently, $\varphi$ establishes a one-to-one correspondence (up to isomorphism) between the class of monadic expansions of $L$ to the class of monadic c-differential residuated distributive lattices expansions of $\mathbf{K(L)}$.
\end{theorem}
\begin{proof}
(1)~~If $(L,\square,\lozenge)$ is a monadic residuated distributive lattice, then by Theorem \ref{theorem4.1}(5),(7),(8) and (14)-(16) we have that  $(\mathbf{K}(L),\square_K)$ satisfies the conditions ($\square_c1$)-($\square_c7$), which shows that $(\mathbf{K}(L),\square_K)$ is a monadic c-differential residuated distributive lattice. 
Let $(L,\square_1,\lozenge_1)$ and $(L,\square_2,\lozenge_2)$ be two monadic residuated distributive lattices such that
\begin{center}$\square_{K_1}(x,y)=(\square_1x,\lozenge_1y)=(\square_2x,\lozenge_2y)=\square_{K_2}(x,y)$
\end{center}
for any $(x,y)\in \mathbf{K}(L)$. In particular, as $(x,{-}x)\in \mathbf{K}(L)$, we have
\begin{center}
$(\square_1x,\lozenge_1(-x))=(\square_2x,\lozenge_2(-x))$,
\end{center}
By Proposition~\ref{proposition3.4}(25),
existential quantifiers are uniquely determined by universal quantifiers. Hence, since 
$\square_1=\square_2$, we conclude that $\lozenge_1=\lozenge_2$.


  (2)~~If $(\mathbf{K}(L),\square_K)$ is a monadic c-difference residuated distributive lattice (by Theorem \ref{theorem4.1} this is in fact the case), then we define the function $\square:L\rightarrow L$ by
\begin{center} $\square x=\pi_1\square_K(x,0)$,
\end{center}
and prove that $(L,\square,\lozenge)$ is a monadic residuated distributive lattice, in which $\lozenge x:=\neg\square \neg x$. Here we first prove that $(L,\square)$ is a monadic involutive residuated distributive lattice. By Proposition \ref{proposition4.5}, we have
\begin{center}~~~$(\blacktriangle)$~~~ $\square_K(x,0)=\square_K((x,0)\cup c)=(\square x,0)$.
\end{center}

(M1):\ By Definition \ref{definition4.3} and $(\blacktriangle)$, we have
\begin{center} $(1,0)=\square_K(x,0)\leadsto(x,0)=(\square x,0)\leadsto (x,0)=(\square x\rightarrow x,0)$,
\end{center}
which implies
\begin{center} $\square x\rightarrow x=1$.
\end{center}

(M3):\ By $(\blacktriangle)$ and Proposition \ref{proposition4.5}(9), we have
\begin{eqnarray*}
               (\square(x\wedge y),0)&=&\square_K(x\cap y,0)\\&=&\square_K(x,0)\cap \square_K(y,0)\\&=&(\square x,0)\cap (\square y,0)\\&=&(\square x\wedge \square y,0),
               \end{eqnarray*} which implies
               \begin{center} $\square(x\wedge y)=\square x\wedge \square y$.
               \end{center}

(M7):\ By $(\blacktriangle)$, we have
\begin{center} $\square(x\leadsto y,0)=(\square x,0)\leadsto (y,0)=\square_K(x,0)\leadsto (y,0)$,
\end{center}
and hence
\begin{eqnarray*}
               (\square(\square x\rightarrow y),0)&=&\square_K(\square_K(x,0)\leadsto (y,0))\\&=&\square_K(x,0)\leadsto \square_K(y,0)\\&=&(\square x,0)\rightarrow (\square y,0)\\&=&(\square x\rightarrow\square y,0),
               \end{eqnarray*}
which proves
\begin{center} $\square(\square x\rightarrow y)=\square x\rightarrow\square y$.
 \end{center}

 Hence by Definition \ref{definition3.5} and Theorem \ref{theorem3.6} that $(L,\square,\lozenge)$ is a monadic residuated distributive lattice.

Moreover, we prove that $\square_K=\square_c$. By Proposition \ref{proposition4.5}, we have
\begin{center} $\square_K(x,0)=(\square x,0)=\square_c(x,0)$,
\end{center}
and
\begin{eqnarray*}
              \square_c(0,y)&=&\square_c(({-}y,0)\otimes c)\\&=& \square_c({-}y,0)\otimes c\\&=&(\square {-}y,0)\otimes c\\&=& (0,\lozenge y)\\&=& \square_K(0,y).
               \end{eqnarray*}
We also show that
\begin{center} $\square_c(x,y)\cup c=\square_c(x,0)=(\square x,0)=\square_K(x,y)\cup c$,

$\square_c(x,y)\cap c=\square_c(x,0)=(0,\lozenge y)=\square_K(x,y)\sqcap c$.
\end{center}
Therefore the result follows by distributivity.
\end{proof}

\subsection{Extending the categorical equivalence between the categories $\mathbb{RDL}$ and $\mathbb{DRDL'}$ }

In this subsection, we will extend the categorical equivalence $\mathbf{C} \dashv \mathbf{K}$. Specifically, here we define two extended functors, which we also call $\mathbf{C}$ and $\mathbf{K}$, between the category $\mathbb{MRDL}$, whose objects are monadic residuated distributive lattices, and the category $\mathbb{MDRDL}$, whose objects are weak monadic c-differential residuated distributive lattices. In both categories, the morphisms are the corresponding algebra homomorphisms.

It is easy to check that the map
\begin{center}
$\mathbf{K}_f:(\mathbf{K}(L_1),\square_{K_1},\lozenge_{K_1})\rightarrow (\mathbf{K}(L_2),\square_{K_2},\lozenge_{K_2})$
 \end{center} given by
 \begin{center} $\mathbf{K}_f(x,y)=(f(x),f(y))$
 \end{center} is indeed a morphism in $\mathbb{MRDL}$ from $(\mathbf{K}(L_1),\square_{K_1},\lozenge_{K_1})$ to $(\mathbf{K}(L_2),\square_{K_2},\lozenge_{K_2})$. These assignments establish a functor $\mathbf{K}$ from $\mathbb{MRDL}$ to $\mathbb{MDRDL}$.

Moreover, if $g:\mathcal{L}_1\rightarrow \mathcal{L}_2$ is a homomorphism of c-differential residuated distributive lattices, then it is not hard to see that
\begin{center}$\mathbf{C}_g:\mathcal{C(L}_1)\rightarrow \mathcal{C(L}_2)$
\end{center} defined by
\begin{center} $\mathbf{C}_g(x)=g(x)$
\end{center} is a homomorphism of residuated distributive lattices.

The above construction can be lifted to the corresponding monadic algebraic structures.

\begin{proposition}\label{proposition5.7.} Let $(\mathcal{L},\square_c)\in \mathbb{MDRDL}$. Then $(\mathcal{C(L)},\square_c,\lozenge_c)\in\mathbb{MRDL}$, where
\begin{center} $\lozenge_cx=\neg\square_c(\neg x)$ for any $x\in \mathcal{L}$.
\end{center}
Moreover, if
$g:\mathcal{L}_1\rightarrow \mathcal{L}_2$  is a morphism in $\mathbb{MDRDL}$, then
$\mathbf{C}_g:\mathcal{C(L}_1)\rightarrow \mathcal{C(L}_2)$ is a morphism in $\mathbb{MRDL}$.
\end{proposition}
\begin{proof} First, we will show that $(\mathcal{C(L)},\square_c,\lozenge_c)\in\mathbb{MRDL}$. The axioms (M1), (M3) and (M7) can be directly obtained from $\square_c2$-$\square_c4$, respectively.
The proofs of (M2) and (T6) are similar to those of (M2) and (T6) in the derivation from Theorem \ref{theorem3.6}, respectively. The moreover part can be obtained from the fact that $\mathbf{C}_g$ is a homomorphism of residuated distributive lattices.
\end{proof}

\begin{proposition}\label{proposition5.12.} Let $(L,\square,\lozenge)$ be a monadic residuated distributive lattice. Then the map
\begin{center} $\phi:(L,\square,\lozenge)\rightarrow (\mathcal{C(\mathbf{K}}(L)),\square_K,\lozenge_K)$
\end{center} given by
\begin{center} $\phi(x)=(x,0)$
\end{center} is an isomorphism in $\mathbb{MRDL}$.
\end{proposition}
\begin{proof} 
It is well known (see \cite{Castiglionib} that $\phi$ is bijective homomorphism of distributive residuated lattices. To prove our claim, it remains to show that it preserves both $\square$ and $\lozenge$. Indeed, let $x \in L$. From (M4) and (M5), it follows that
\begin{center} $\phi(\square x)=(\square x,0)=(\square x,\lozenge 0)=\square_K((x,0))=\square_K(\phi(x))$,

$\phi(\lozenge x)=(\lozenge x,0)=(\lozenge x, \square 0)=\lozenge_K((x,0))=\lozenge_K(\phi(x))$,

\end{center} which proves that $\phi$ preserves the universal and existential quantifiers.
\end{proof}

If $\mathcal{L}$ is a c-differential residuated distributive lattice, then
\begin{center} $\psi:\mathcal{L}\rightarrow  \mathbf{K}(\mathcal{C(L)})$
\end{center} given by
\begin{center} $\psi(x)=(x\cup c,\neg x\cup c)$
\end{center} is an injective homomorphism of c-differential residuated distributive lattices (see \cite{Castiglionib}).

To obtain the main result of this subsection, we need the following important results.

\begin{proposition}\label{proposition5.13} Let $(\mathcal{L},\square_c)$ be a monadic c-differential residuated distributive lattice. Then the injective map $\psi$ preserves $\square_c$.
\end{proposition}
\begin{proof}
First, note that $\lozenge_c c = c$. Let $x \in \mathcal{L}$. From Definition \ref{definition4.3}, the definition of $\lozenge_c$ and Proposition \ref{proposition3.4}(15), we have:
\begin{eqnarray*}
\psi(\square_c x) &=& (\square_c x \cup c, \neg(\square_c x) \cup c) \\
&=& (\square_c (x \cup c), \lozenge_c\neg x \cup c) \\
&=& (\square_c(x \cup c), \lozenge_c(\neg x \cup c)) \\
&=& \square_K(x \cup c, \neg x \cup c) \\
&=& \square_K(\psi(x)).
\end{eqnarray*}
Thus,
\begin{center}
$\psi(\square_c x) = \square_K(\psi(x))$.
\end{center}
This shows that the map $\psi$ preserves the center universal quantifier.
\end{proof}

\begin{proposition}\label{proposition5.14}\cite{Castiglionib}  A c-differential residuated distributive lattice $\mathcal{L}$ satisfies the condition $\mathbf{CK}$ if and only if $\psi$ is a surjective.
\end{proposition}

Here we denote by $\mathbb{MDRDL'}$ the full subcategory of $\mathbb{MDRDL}$ whose objects satisfy the condition $\mathbf{CK}$. Now, as a consequence of the previous results, we are ready to present the main result of the paper.
\begin{theorem}\label{theorem5.15} The functors $\mathbf{K}$ and $\mathbf{C}$ establish a categorical equivalence between $\mathbb{MRDL}$ and $\mathbb{MDRDL'}$ with the natural isomorphisms $\phi$ and $\psi$.
\end{theorem}
\begin{proof} It follows from Propositions \ref{proposition5.12.}, \ref{proposition5.13} and \ref{proposition5.14}.
\end{proof}

\begin{corollary}\label{corollary5.16}
    The functor $\mathbb{MRDL}\vdash \mathbb{RDL}$ followed by $\mathbf{K}:\mathbb{RDL}\vdash \mathbb{DRDL'}$ factors through the forgetful $\mathbf{F}:\mathbb {MDRDL'}\vdash  \mathbb{DRDL'}$.
\end{corollary}
\begin{proof} The proof is straightforward.
\end{proof}

As a consequence of Theorem \ref{theorem5.15} and Corollary \ref{corollary5.16}, we obtain a commutative diagram:
\begin{center}
\begin{tikzpicture}
  \node[inner sep=2pt] (MV^bullet) at (0,0) {$\mathbb{RDL}$};
  \node[inner sep=2pt] (MV) at (0,3) {$\mathbb{MRDL}$ };
  \node[inner sep=2pt] (MDRL) at (3,0) {$\mathbb{DRDL'}$};
  \node[inner sep=2pt] (iIRL_0) at (3,3) {$\mathbb{MDRDL'}$};  
\draw[->] ([yshift=3pt]MV.east) -- ([yshift=3pt]iIRL_0.west) node[midway, above=3pt] {$\mathbf{K}$};
\draw[<-] ([yshift=-3pt]MV.east) -- ([yshift=-3pt]iIRL_0.west) node[midway, below=3pt] {$\mathbf{C}$};
\draw[->] ([yshift=3pt]MDRL.west) -- ([yshift=3pt]MV^bullet.east) node[midway, above=3pt] {$\mathbf{K}$};
\draw[<-] ([yshift=-3pt]MDRL.west) -- ([yshift=-3pt]MV^bullet.east) node[midway, below=3pt] {$\mathbf{C}$};

\draw[->] (MV) -- (MV^bullet) node[midway, left=2pt] {$\mathbf{F}$};
  \draw[->] (iIRL_0) -- (MDRL) node[midway, right=2pt] {$\mathbf{F}$};
\end{tikzpicture} 
\begin{center}  {\bf{Figure 4.}} Diagram of the functors $\mathbf{K}$ and $\mathbf{C}$ relate the categories $\mathbb{MRDL}$ and $\mathbb{MDRDL'}$ .
\end{center}
\end{center}
with the properties discussed in the fifth paragraph of Introduction.\medskip

 Castiglioni, Lewin and Sagastume proved in \cite{Castiglionic} that the category $\mathbb{MV}^\bullet$, whose objects are MV$^\bullet$-algebras, is the image of $\mathbb{MV}$, whose objects are MV-algebras, by the equivalence $\mathbf{K}^\bullet$. Following the line, Sagatume and San Mart\'{i}n lifted the $\mathbf{K}^\bullet$ between the categories $\mathbb{MV}_\textrm{U}$, whose objects are UMV-algebras, and $\mathbb{MV}^\bullet_\textrm{cU}$, whose objects are pairs formed by an object of $\mathbf{K}^\bullet$ and a cU-operator.
 
  As a directly consequence of Corollary \ref{corollary3.8}(2), we have the following result.

\begin{corollary}\label{corollary5.17} The subcategory of $\mathbb{MRDL}$ of monadic involutive residuated lattices determined by the equations (INV), (DIV) and (PRE) is term-equivalent to the category $\mathbb{MV}_\textrm{U}$ of UMV-algebras.
\end{corollary} 

As a consequence of Figure 1, Figure 4 and Corollary \ref{corollary5.17},  In the following diagram we have the relationship among the above mentioned categories, where \emph{inc} denotes the inclusion functor:

\begin{center}
\begin{tikzpicture}
  \node[inner sep=2pt] (MV^bullet) at (0,0) {$\mathbb{MV}^\bullet_\textrm{cU}$};
  \node[inner sep=2pt] (MV) at (0,3) {$\mathbb{MV}_\textrm{U}$};
  \node[inner sep=2pt] (MDRL) at (3,0) {$\mathbb{MDRDL'}$};
  \node[inner sep=2pt] (iIRL_0) at (3,3) {$\mathbb{MRDL}$};
  \node[inner sep=2pt] (DRLyipie) at (6,0) {$\mathbb{DRDL'}$};
  \node[inner sep=2pt] (IRL_0) at (6,3) {$\mathbb{RDL}$};
  
  \draw[->] ([xshift=3pt]MV.south) -- ([xshift=3pt]MV^bullet.north) node[midway, right=3pt] {$\mathscr{K}$};
  \draw[<-] ([xshift=-3pt]MV.south) -- ([xshift=-3pt]MV^bullet.north) node[midway, left=3pt] {$\mathbf{K}^\bullet$};
  \draw[->] ([xshift=3pt]MDRL.north) -- ([xshift=3pt]iIRL_0.south) node[midway, right=3pt] {$\mathbf{C}$};
  \draw[<-] ([xshift=-3pt]MDRL.north) -- ([xshift=-3pt]iIRL_0.south) node[midway, left=3pt] {$\mathbf{K}$};
  \draw[->] ([xshift=3pt]DRLyipie.north) -- ([xshift=3pt]IRL_0.south) node[midway, right=3pt] {$\mathbf{C}$};
  \draw[<-] ([xshift=-3pt]DRLyipie.north) -- ([xshift=-3pt]IRL_0.south) node[midway, left=3pt] {$\mathbf{K}$};

  \draw[->] (MV) -- (iIRL_0) node[midway, above=2pt] {inc};
  \draw[->] (iIRL_0) -- (IRL_0) node[midway, above=2pt] {$\mathbf{F}$};
  \draw[->] (MV^bullet) -- (MDRL) node[midway, above=2pt] {inc};
  \draw[->] (MDRL) -- (DRLyipie) node[midway, above=2pt] {$\mathbf{F}$};
\end{tikzpicture}
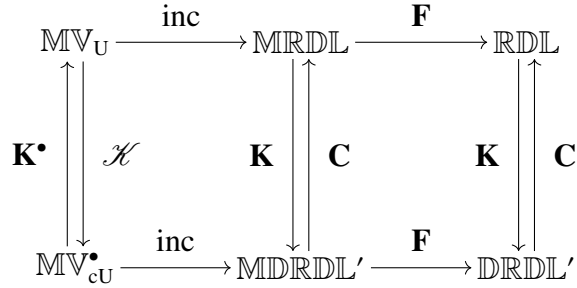
\begin{center}  {\bf{Figure 5.}} Diagram of the relationship between the categories $\mathbb{MV}_\textrm{U}$, $\mathbb{MV}^\bullet_\textrm{cU}$, $\mathbb{MDRDL'}$ and $\mathbb{MRDL}$.
\end{center}
\end{center}

\section{A contextual translation}

Let $\mathcal V$ and $\mathcal W$ be varieties, and denote by the same symbols the
corresponding categories of algebras and homomorphisms. In~\cite{M2018}, Moraschini
showed that there is a correspondence between nontrivial adjunctions between
$\mathcal V$ and $\mathcal W$---that is, pairs of functors
$F : \mathcal V \to \mathcal W$ and $G : \mathcal W \to \mathcal V$ such that $F$
is left adjoint to $G$ (written $F \dashv G$)---and nontrivial
$\kappa$-contextual translations from the equational consequence relation
associated with $\mathcal V$ to that associated with $\mathcal W$, denoted by
$\models_{\mathcal V}$ and $\models_{\mathcal W}$, respectively. Here, $\kappa$
is an arbitrary cardinal, not necessarily finite. Well-known instances of
$\kappa$-contextual translations include the Kolmogorov translation arising from
the classical adjunction induced by the Kalman construction for distributive
lattices~\cite{Cignoli}, the 2-contextual translation determined by the Kalman construction connecting tense  integral commutative
residuated distributive lattices and tense c-differential residuated lattices \cite{CPZ2024}, the Glivenko translation determined by Glivenko’s functor
connecting Boolean algebras and Heyting algebras~\cite{Kol}, and the Gödel
translation relating the equational consequence for Heyting algebras and interior
algebras~\cite{G1933}.

The purpose of this section is to show that Theorem \ref{theorem5.15} can be applied to obtain a
finite, nontrivial $2$-contextual translation between the equational consequence
relations of monadic c-differential residuated distributive lattices satisfying {\bf{CK}} and monadic residuated distributive lattices.

We begin by fixing some notation. Let $X$ be a set of variables. For a
propositional language $L$, we denote by $\mathrm{Fm}_L(X)$ the absolutely free
algebra of $L$-terms over $X$, and by $\mathrm{Fm}_L(X)$ its underlying set. In
particular, for a cardinal $\lambda$, $\mathrm{Fm}_L(\lambda)$ denotes the set of
$L$-terms in the variables $\{x_i : i < \lambda\}$. We define
\[
\mathrm{Eq}(L,X) := \mathrm{Fm}_L(X) \times \mathrm{Fm}_L(X).
\]
If $\mathcal V$ is a variety, $\mathbf F_{\mathcal V}(X)$ stands for the free
$\mathcal V$-algebra over $X$. Moreover, if $L$ is the language of $\mathcal V$
and $t(x_1,\dots,x_n)$ is an $L$-term, we use the same notation to denote its image
under the canonical homomorphism from $\mathrm{Fm}_L(X)$ onto
$\mathbf F_{\mathcal V}(X)$. When $X = \{x_1,\dots,x_k\}$, we write
$\mathbf F_{\mathcal V}(k)$ instead of $\mathbf F_{\mathcal V}(X)$. We also write
$\mathcal V$ for the category determined by the variety. Given $A \in \mathcal V$
and $S \subseteq A \times A$, $\mathrm{Cg}_A(S)$ denotes the congruence generated by
$S$. For tuples $\vec a, \vec b \in A^N$, we write $\mathrm{Cg}_A(\vec a,\vec b)$
for the congruence generated by the pairs
$(a_1,b_1),\dots,(a_N,b_N)$.

\subsection*{ $2$-contextual translations}

Let $\mathcal K$ be a class of similar algebras and let $L_{\mathcal K}$ be its
language. The language $L_{\mathcal K}^2$ consists of all pairs of
$L_{\mathcal K}$-terms built from the variables
\[
X_n = \{x_j^1 : 1 \le j \le n\} \cup \{x_j^2 : 1 \le j \le n\},
\]
for each $n \in \mathbb N$. That is, an $n$-ary operation symbol in
$L_{\mathcal K}^2$ has the form
\[
f\big((x_1^1,x_1^2),\dots,(x_n^1,x_n^2)\big)
=
\big(
t_1(x_1^1,x_1^2,\dots,x_n^1,x_n^2),
t_2(x_1^1,x_1^2,\dots,x_n^1,x_n^2)
\big),
\]
where $t_1,t_2 \in \mathrm{Fm}_{L_{\mathcal K}}(X_n)$.
For example, if $L_{\mathcal K}=\{\to,\vee\}$, the operation $\oplus$ defined by
\[
\oplus\big((x_1^1,x_1^2),(x_2^1,x_2^2)\big)
:=
\big((x_2^2 \to (x_1^1 \to x_2^1))\vee x_1^2,\, x_1^2 \vee x_2^1\vee x_2^2\big)
\]
is a binary operation of $L_{\mathcal K}^2$.

For any $A \in \mathcal K$, let $A[2]$ be the algebra of type $L_{\mathcal K}^2$
with universe $A^2$, where each operation is interpreted componentwise according
to the defining terms. The class
\[
\mathcal K[2] := \mathbb I\{A[2] : A \in \mathcal K\}
\]
is called the $2$-matrix power of $\mathcal K$. It is known that $\mathcal K[2]$
is a variety if and only if $\mathcal K$ is one (Theorem~2.3(iii) of~\cite{Mc1996}).
Moreover, any homomorphism $h : A \to B$ induces a homomorphism
$h \times h : A^2 \to B^2$, which we denote by $h[2]$. This yields a functor
$[2] : \mathcal K \to \mathcal K[2]$.

Let $L_{\mathcal V}$ and $L_{\mathcal W}$ be the languages of the varieties
$\mathcal V$ and $\mathcal W$, respectively. A $2$-translation from
$L_{\mathcal V}$ to $L_{\mathcal W}$ is a mapping
$\tau : L_{\mathcal V} \to L_{\mathcal W}^2$ preserving arities. Such a
translation extends uniquely to terms, giving rise to a map
\[
\tau^\ast :
\mathrm{Fm}_{L_{\mathcal V}}(\lambda)
\longrightarrow
\mathrm{Fm}_{L_{\mathcal W}}(X_\lambda)^2,
\]
defined recursively on variables, constants, and compound terms. This extension
further induces a transformation on sets of equations.

A $2$-contextual translation of $\models_{\mathcal V}$ into
$\models_{\mathcal W}$ is a pair $(\tau,\Theta)$, where $\tau$ is a
$2$-translation and $\Theta \subseteq \mathrm{Eq}(L_{\mathcal W},2)$ is a set of
equations satisfying suitable preservation conditions with respect to operations
and entailment. The set $\Theta$ is called the context of the translation. The
translation is said to be nontrivial if there exists a nonempty pair of constant
terms witnessing that the context does not collapse all variables to equality (see \cite{M2018} for details).

We now provide an explicit description of the algebras obtained by applying the
functor $\mathbf{C}$ to finitely generated free $\mathbb{MDRDL'}$-algebras.
Specifically, we show that for each $n\geq 1$, the algebra
$\mathbf{C}(\mathbf{F}_{\mathbb{MDRDL'}}(n))$ is isomorphic to a quotient of the
free $\mathbb{MRDL}$-algebra on $2n$ generators, where the defining congruence
forces, for every generator, the product of the corresponding pair of variables
to be equal to $0$. This representation will be instrumental in the construction
of the associated contextual $2$-translation.

\begin{theorem}\label{theorem C}
The algebra $\mathbf{C}(\mathbf{F}_{\mathbb{MDRDL'}}(1))$ is isomorphic to
\[
\mathbf{F}_{\mathbb{MRDL}}(x_{1}^{1},x_{1}^{2})\big/
\mathrm{Cg}_{\mathbf{F}_{\mathbb{MRDL}}(2)}(x_{1}^{1}\odot x_{1}^{2}, 0).
\]
Moreover, if $X_n=\{x_{1}^{1},x_{1}^{2},\ldots,x_{n}^{1},x_{n}^{2}\}$, then
$\mathbf{C}(\mathbf{F}_{\mathbb{MDRDL'}}(n))$ is isomorphic to
$\mathbf{F}_{\mathbb{MRDL}}(X_n)/\theta_n$, where
\[
\theta_n \;=\;
\bigvee_{i=1}^{n}
\mathrm{Cg}_{\mathbf{F}_{\mathbb{MRDL}}(X_n)}(x_i^{1}\odot x_i^{2}, 0),
\]
for every $n\geq 1$.
\end{theorem}
\begin{proof}
    Observe that the first part is an adaptation of Theorem~7 of~\cite{CPZ2024}.
For the \emph{moreover} part, note that the case $n=2$ is an adaptation of
Corollary~5 of \cite{CPZ2024}. We proceed by induction on $n$.
Assume, as induction hypothesis, that
\[
\mathbf{C}(\mathbf{F}_{\mathbb{MDRDL'}}(k))
\cong
\mathbf{F}_{\mathbb{MRDL}}(X_k)/\theta_k.
\]
We show that the statement holds for $n=k+1$.
Indeed, since both functors $\mathbf{C}$ and $\mathbf{F}_{\mathbb{MDRDL'}}$
preserve coproducts, and
$\mathbf{C}(\mathbf{F}_{\mathbb{MDRDL'}}(1))$ is isomorphic to
\[
\mathbf{F}_{\mathbb{MRDL}}(x_{k+1}^{1},x_{k+1}^{2})
\big/
\mathsf{Cg}^{\mathbf{F}_{\mathbb{MRDL}}(2)}(x_{k+1}^{1}\odot x_{k+1}^{2}, 0),
\]
it follows from the well-known description of coproducts of algebras that
\begin{displaymath}
\begin{array}{rcl}
\mathbf{C}(\mathbf{F}_{\mathbb{MDRDL'}}(k+1))
& \cong &
\mathbf{C}(\mathbf{F}_{\mathbb{MDRDL'}}(k))
\,+\,
\mathbf{C}(\mathbf{F}_{\mathbb{MDRDL'}}(1)) \\[0.4em]
& \cong &
\mathbf{F}_{\mathbb{MRDL}}(X_k)/\theta_k
\,+\,
\mathbf{F}_{\mathbb{MRDL}}(x_{k+1}^{1},x_{k+1}^{2})
\big/
\mathsf{Cg}^{\mathbf{F}_{\mathbb{MRDL}}(2)}(x_{k+1}^{1}\odot x_{k+1}^{2}, 0) \\[0.4em]
& \cong &
\mathbf{F}_{\mathbb{MRDL}}(X_{k+1})
\big/
\mathrm{Cg}_{\mathbf{F}_{\mathbb{MRDL}}(X_{k+1})}\!\big(
Z\cup\,
\{(x^{1}_{k+1}\odot x^{2}_{k+1},0)\}
\big) \\[0.4em]
& \cong &
\mathbf{F}_{\mathbb{MRDL}}(X_{k+1})/\theta_{k+1},
\end{array}
\end{displaymath} 
where $Z=\{(x^{1}_{i}\cdot x^{2}_{i},0)\mid 1\leq i\leq k\}$. Thus, the statement holds for every $n\geq 1$, as claimed.
\end{proof}

We now describe the construction of the contextual $2$-translation associated
with the adjunction $\mathbf{C} \dashv \mathbf{K}$. Consider the homomorphism
$h\colon \mathbf{F}_{\mathbb{MRDL}}(2)\to
\mathbf{C}(\mathbf{F}_{\mathbb{MDRDL'}}(1))$ given by
\[
t(x_{1}^{1},x_{1}^{2})
\longmapsto
t^{\mathbf{C}(\mathbf{F}_{\mathbb{MDRDL'}}(1))}(z\cup c,\sim z\cup c),
\]
and let $\psi$ be an $n$-ary function symbol in the language of monadic
$c$-differential residuated distributive lattices satisfying \textbf{CK}.
It is well known that $\psi$ can be identified with a unique
$\mathbb{MDRDL'}$-homomorphism
$\psi\colon \mathbf{F}_{\mathbb{MDRDL'}}(1)\to
\mathbf{F}_{\mathbb{MDRDL'}}(n)$.
Let
$X_n=\{x_{1}^{1},x_{1}^{2},\ldots,x_{n}^{1},x_{n}^{2}\}$.
By Theorem~4.3 of~\cite{M2018}, the associated $2$-translation $\tau$ is
determined by a homomorphism $\tau(\psi)$ that makes the following diagram
commute:
\begin{displaymath}
\xymatrix{
& & \ar@{-->}[ddll]_-{\tau(\psi)} \mathbf{F}_{\mathbb{MRDL}}(2) \ar[d]^-{h} \\
& & C(\mathbf{F}_{\mathbb{MDRDL'}}(1)) \ar[d]^-{C(\psi)} \\
\mathbf{F}_{\mathbb{MDRDL'}}(X_n) \ar[rr]_-{\pi_n} & & C(\mathbf{F}_{\mathbb{MDRDL'}}(n))
}
\end{displaymath}
where $\pi_n$ denotes the canonical homomorphism induced by Theorem~\ref{theorem C}. We emphasize that the existence of $\tau(\psi)$ follows from the fact that
$\mathbf{F}_{\mathbb{MRDL}}(2)$ is onto-projective in the variety $\mathsf{\mathbb{MRDL}}$.
Moreover, the map
$\tau \colon L_{\mathbb{MDRDL'}}\rightarrow L_{\mathbb{MRDL}}^{2}$ is completely
determined by its action on the generators. That is,
$\tau(\psi) = (\tau(\psi)(x_{1}^{1}), \tau(\psi)(x_{1}^{2}))$.
Consequently, to define $\tau(\psi)$ it suffices to find terms
$t(x_{1}^{1},x_{1}^{2},\ldots,x_{n}^{1},x_{n}^{2})$ and
$s(x_{1}^{1},x_{1}^{2},\ldots,x_{n}^{1},x_{n}^{2})$ in
$\mathbf{F}_{\mathbb{MDRDL'}}(X_n)$ such that
\begin{equation}\label{eq: terms}
C(\psi)(h(x)) = h(t(x_{1}^{1},x_{1}^{2},\ldots,x_{n}^{1},x_{n}^{2}))
\quad\text{and}\quad
C(\psi)(h(y)) = h(s(x_{1}^{1},x_{1}^{2},\ldots,x_{n}^{1},x_{n}^{2})).    
\end{equation}

An important aspect of the 2-translation $\tau$ is that the associated terms $t$ and $s$ admit an explicit construction via the functor $\mathbf{K}$, as the following result shows.

\begin{lemma}\label{lem: 2-translation}
The adjunction $\mathbf{C}\dashv \mathbf{K}$ induces the 2-translation  $\tau \colon L_{\mathbb{MDRDL'}}\rightarrow L_{\mathbb{MRDL}}^{2}$ defined as follows:
\begin{displaymath}
\begin{array}{rcl}
\tau(\cap)((x_{1}^{1},x_{1}^{2}), (x_{2}^{1},x_{2}^{2})) & := & (x_{1}^{1}\wedge x_{2}^{1},x_{1}^{2}\vee x_{2}^{2})\\
\tau(\cup)((x_{1}^{1},x_{1}^{2}), (x_{2}^{1},x_{2}^{2})) & := & (x_{1}^{1}\vee x_{2}^{1},x_{1}^{2}\wedge x_{2}^{2})\\
\tau(\otimes)((x_{1}^{1},x_{1}^{2}), (x_{2}^{1},x_{2}^{2})) & := & (x_{1}^{1}\odot x_{2}^{1}, (x_{1}^{1}\rightarrow x_{2}^{2})\wedge (x_{2}^{1}\rightarrow x_{1}^{2}))\\
\tau(\leadsto)((x_{1}^{1},x_{1}^{2}), (x_{2}^{1},x_{2}^{2})) & := & ((x_{1}^{1}\rightarrow x_{2}^{1})\wedge (x_{2}^{2}\rightarrow x_{1}^{2}), x_{1}^{1}\cdot x_{2}^{2})\\
\tau(\square_c)(x_{1}^{1},x_{1}^{2}) & := & (\square x_{1}^{1},\lozenge x_{1}^{2})\\
\tau({\sim})(x_{1}^{1},x_{1}^{2}) & := & (x_{1}^{2},x_{1}^{1})\\
\tau(c) & := & (0,0)\\
\tau(0) & := & (0,1)\\
\tau(1) & := & (1,0).\\
\end{array}
\end{displaymath}    
\end{lemma}
\begin{proof}
Recall first that the map
$\mathbf{C}(\square_c)\colon
\mathbf{C}(\mathbf{F}_{\mathbb{MDRDL'}}(1))
\to
\mathbf{C}(\mathbf{F}_{\mathbb{MDRDL'}}(1))$
sends a term $u(z)$ to $u(\square_c(z))$. Moreover, by
Theorem~\ref{theorem C}, the homomorphism
$h\colon \mathbf{F}_{\mathbb{MRDL}}(2)\to
\mathbf{C}(\mathbf{F}_{\mathbb{MDRDL'}}(1))$
maps a term $t(x_{1}^{1},x_{1}^{2})$ to
$t^{\mathbf{C}(\mathbf{F}_{\mathbb{MDRDL'}}(1))}(z\cup c,\sim z\cup c)$.
Now consider the terms
$t(x_{1}^{1},x_{1}^{2})=\square x_{1}^{1}$ and
$s(x_{1}^{1},x_{1}^{2})=\lozenge x_{1}^{2}$.
It is immediate to verify that
$h(\square_c(x_{1}^{1}))=\mathbf{C}(\square_c)(h(x_{1}^{1}))$.
On the other hand, by Proposition~\ref{proposition5.7.}, we obtain
\[
\begin{array}{rcl}
h(\lozenge x_{1}^{2})
& = & \lozenge(\sim z \cup c) \\
& = & \sim \square(z \cap c) \\
& = & \sim(\square(z)\cap c) \\
& = & \sim \square(z)\cup c \\
& = & \mathbf{C}(\square)(\sim z\cup c) \\
& = & \mathbf{C}(\square)(h(x_{1}^{2})).
\end{array}
\]
Therefore, we may define
$\tau(\square_c)(x_{1}^{1},x_{1}^{2})=(\square x_{1}^{1},\lozenge x_{1}^{2})$,
as required. The proofs for the clauses concerning
$\tau(\cap)$, $\tau(\cup)$, $\tau(\otimes)$, and $\tau(\leadsto)$
are analogous to those presented in Lemma~14 of~\cite{CPZ2024}.
\end{proof}

By Theorem~\ref{theorem C}, the generating set $\{(x\odot y,0)\}$ may be identified
with the set of equations $\Theta=\{x\odot y\approx 0\}$. It then follows from
Theorem~4.3 of~\cite{M2018} that the pair $(\tau,\Theta)$ yields a $2$-contextual
translation from $\models_{\mathbb{MDRDL'}}$ to $\models_{\mathbb{MRDL}}$.
Since the functor $\mathbf{C}$ is faithful (Theorem~\ref{theorem5.15}),
Lemma~6.4 of~\cite{M2016} applies, yielding the following theorem.

\begin{theorem}
For every cardinal $\lambda$ and
$\Phi \cup \{ \varepsilon \approx \delta\}\subseteq
\mathrm{Eq}(\mathcal{L}_{\mathbb{MDRDL'}},\lambda)$,
\[
\Phi \models_{\mathbb{MDRDL'}} \varepsilon \approx \delta
\Longleftrightarrow
\tau_{\ast}(\Phi)\cup
\{(x_j^{1}\odot x_j^{2},x_j^{1}\odot x_j^{2})\approx (0,0)\colon j<\lambda \}
\models_{\mathbb{MRDL}}
\tau_{\ast}(\varepsilon \approx \delta).
\]
Here $(\tau, \Theta)$ denotes the contextual translation of
$\models_{\mathbb{MDRDL'}}$ into $\models_{\mathbb{MRDL}}$ induced by $\mathbf{C}$.
\end{theorem}

In informal terms, the theorem shows that equational reasoning in
$\mathbb{MDRDL'}$ can be faithfully carried out inside $\mathbb{MRDL}$ by
duplicating variables. That is to say, everything that can be deduced in
$\mathbb{MDRDL'}$ is precisely what can be deduced in $\mathbb{MRDL}$ via
$\tau$, modulo the equations $x_j^{1}\odot x_j^{2}\approx 0$.

\section{Conclusions and further work}

 The motivation of this paper is to lift the relationship between the categories $\mathbb{RDL}$ and  $\mathbb{DRDL'}$ to the categories $\mathbb{MRDL}$ and  $\mathbb{MDRDL'}$. In order to achieve our aim, we first reviewed some of their algebraic results related to monadic residuated lattices and obtained the notion of monadic involutive residuated lattices. The results of Subsection 4.1 established a bijection that maps a monadic residuated distributive lattice into a monadic c-differential residuated distributive lattice that satisfy $\mathbf{CK}$. This map extends to a functor that is a categorical equivalence in Subsection 4.2. Subsection 4.2 showed that there exists a categorical equivalence between the categories $\mathbb{MRDL}$ and  $\mathbb{MDRDL'}$.

 The paper is devoted to transferring the balance between the categories $\mathbb{MRDL}$ and  $\mathbb{MDRDL'}$. There is also an open problem that establish the corresponding relation between the categories $\mathbb{MDFL}_{\textmd{e}}$ and $\mathbb{MCDFL'}_{\textmd{e}}$, whose objects are monadic distributive FL$_\textrm{e}$-algebras and monadic distributive c-differential FL$_\textrm{e}$-algebras with a special condition.
 The critical point for obtaining this more abroad and more common result is to reconstruct the categorical equivalence between the categories $\mathbb{DFL}_{\textmd{e}}$ and $\mathbb{CDFL'}_{\textmd{e}}$, because Kalman functor $\mathbf{K}$ is no longer suitable for these two categories in general, which is one of the topics in our next discussions. Once we did that, we established the relationship among the categories, as shown in the following diagram, where \emph{inc} denotes the inclusion functor and $?_2$ is left adjoint to $?_1$ :

\begin{center}
\begin{tikzpicture}
  \node[inner sep=2pt] (MV^bullet) at (0,0) {$\mathbb{MDRDL'}$};
  \node[inner sep=2pt] (MV) at (0,3) {$\mathbb{MRDL}$};
  \node[inner sep=2pt] (MDRL) at (3,0) {$\mathbb{MCDFL'}_{\textmd{e}}$};
  \node[inner sep=2pt] (iIRL_0) at (3,3) {$\mathbb{MDFL}_{\textmd{e}}$};
  \node[inner sep=2pt] (DRLyipie) at (6,0) {$\mathbb{CDFL'}_{\textmd{e}}$};
  \node[inner sep=2pt] (IRL_0) at (6,3) {$\mathbb{DFL}_{\textmd{e}}$};
  
  \draw[->] ([xshift=3pt]MV.south) -- ([xshift=3pt]MV^bullet.north) node[midway, right=3pt] {$\mathbf{C}$};
  \draw[<-] ([xshift=-3pt]MV.south) -- ([xshift=-3pt]MV^bullet.north) node[midway, left=3pt] {$\mathbf{K}$};
  \draw[->] ([xshift=3pt]MDRL.north) -- ([xshift=3pt]iIRL_0.south) node[midway, right=3pt] {$?_2$};
  \draw[<-] ([xshift=-3pt]MDRL.north) -- ([xshift=-3pt]iIRL_0.south) node[midway, left=3pt] {$?_1$};
  \draw[->] ([xshift=3pt]DRLyipie.north) -- ([xshift=3pt]IRL_0.south) node[midway, right=3pt] {$?_2$};
  \draw[<-] ([xshift=-3pt]DRLyipie.north) -- ([xshift=-3pt]IRL_0.south) node[midway, left=3pt] {$?_1$};

  \draw[->] (MV) -- (iIRL_0) node[midway, above=2pt] {inc};
  \draw[->] (iIRL_0) -- (IRL_0) node[midway, above=2pt] {$\mathbf{F}$};
  \draw[->] (MV^bullet) -- (MDRL) node[midway, above=2pt] {inc};
  \draw[->] (MDRL) -- (DRLyipie) node[midway, above=2pt] {$\mathbf{F}$};
\end{tikzpicture}
\end{center}

As an application of the adjunction $\mathbf{C}\dashv\mathbf{K}$ we obtained a
a finite, nontrivial $2$-contextual translation from the equational consequence
relation of $\mathbb{MDRDL'}$ into that of $\mathbb{MRDL}$. The key step was an
explicit description of the algebras
$\mathbf{C}(\mathbf{F}_{\mathbb{MDRDL'}}(n))$ as quotients of free
$\mathbb{MRDL}$-algebras on duplicated generators, where the additional
congruence enforces the constraint $x_i^{1}\odot x_i^{2}=0$. This representation
allowed us to construct the translation $\tau$ explicitly and to show that
equational reasoning in $\mathbb{MDRDL'}$ can be faithfully simulated inside
$\mathbb{MRDL}$ by duplicating variables and imposing these equations. Beyond its
technical role, this result illustrates how adjunctions provide a uniform and
concrete mechanism to compare equational theories, suggesting further
applications of contextual translations to other algebraic expansions and
logical extensions.

\medskip
\medskip
\noindent{\bf Acknowledgments.} This work is supported by the National Natural Science Foundation of China (12501647), the Humanities and Social Sciences Youth Foundation of Ministry of Education of China "A propositional calculus formal study of monadic substructural logics''(24XJC72040001), and the Natural Science Basic Research Plan in Shaanxi Province of China (2025JC-YBMS-034, 2025JC-YBQN-092), and the Scientific Research Program Funded by Education Department of Shaanxi Provincial Government (Youth Innovation Team of Shaanxi Universities) (23JP132). 

\medskip
\medskip
\noindent{\bf Declarations}\\
{\bf Conflict of interest} The authors declare that there is no conflict of interests.

\end{document}